\documentclass[10pt]{amsart}
\usepackage{amsmath,amsthm,amssymb,url}
\include{xy}
\xyoption{all}
\SelectTips{cm}{}

\newcommand\Z{{\mathbb Z}}
\newcommand\Zpos{\Z_{\ge1}}
\newcommand\Zneg{\Z_{\le-1}}
\newcommand\Znn{\Z_{\ge0}}
\newcommand\ol[1]{\overline{#1}}
\newcommand\Q{{\mathbb Q}}

\newcommand\Qnn{\Q_{\ge0}}

\newcommand\R{{\mathbb R}}

\newcommand\C{{\mathbb C}}

\newcommand\SetT{{\mathcal T}}
\newcommand\F{{\mathbb F}}

\newcommand\lset{\{\,}
\newcommand\rset{\,\}}
\newcommand\lra{\longrightarrow}
\newcommand\bs{\backslash}
\newcommand\inv{^{-1}}
\newcommand\tr{\operatorname{tr}}

\newcommand\eg{{\em e.g.}}
\newcommand\ie{{\em i.e.}}

\newcommand\ip[2]{\langle #1,#2\rangle} 

\newcommand\siX[1]{{\mathcal X}_{#1}} 
\newcommand\Xtwo{\siX2(1)}
\newcommand\Xtwosemi{\Xtwo^{\rm semi}}
\newcommand\XtwoN{\siX2(N)}
\newcommand\XtwoNsemi{\XtwoN^{\rm semi}}

\newcommand\SL[2]{\operatorname{SL}_{#1}(#2)}
\newcommand\SLtwoZ{\SL2\Z}

\newcommand\GL[2]{\operatorname{GL}_{#1}(#2)}
\newcommand\GLtwoZ{\GL2\Z}
\newcommand\GLtwoR{\GL2\R}

\newcommand\Spgp{\operatorname{Sp}}
\newcommand\Sp[2]{\operatorname{Sp}_{#1}(#2)}
\newcommand\SptwoZ{\Sp2\Z}
\newcommand\SptwoQ{\Sp2\Q}

\newcommand\SptwoR{\Sp2\R}
\newcommand\Gamzero{\Gamma_{\!0}}
\newcommand\GzN{\Gamzero(N)}

\newcommand\Ptwoone{\operatorname{P}_{2,1}}

\newcommand\Ptwozero{\operatorname{P}_{2,0}}
\newcommand\PtwozeroQ{{\rm P}_{2,0}(\Q)}

\newcommand\cmat[4]{\left[\begin{array}{cc}
{#1}&{#2}\\{#3}&{#4}\end{array}\right]}
\newcommand\smallmat[4]{\left[\begin{smallmatrix}
{#1}&{#2}\\{#3}&{#4}\end{smallmatrix}\right]}

\newcommand\smallmatabcd{\smallmat abcd}

\newcommand\matoplus{\boxplus}

\newcommand\paramodulargroup[1]{{\rm K}(#1)}
\newcommand\KN{\paramodulargroup N}
\newcommand\KNq{\paramodulargroup{Nq}}
\newcommand\pvar{{\cmat\tau zz\omega}} 
\newcommand\smallpvar{{\smallmat\tau zz\omega}}

\newcommand\fcJ[2]{c(#1;#2)} 
\newcommand\GamsupzeroN{\Gamma^0(N)}
\newcommand\GampmsupzeroN{\Gamma_\pm^0(N)}

\newcommand\UHP{{\mathcal H}} 
\newcommand\UHPtwo{{\UHP_2}}

\newcommand\wtvar[2]{[#1]_{#2}}
\newcommand\wtk[1]{\wtvar{#1}k}

\newcommand\fc[2]{a(#1;#2)} 
\newcommand\e{{\rm e}} 

\newcommand\MFsnoweight{{\mathcal M}} 
\newcommand\MFswtgp[2]{\MFsnoweight_{#1}(#2)}
\newcommand\MFs[1]{\MFswtgp k{#1}} 
\newcommand\CFsnoweight{{\mathcal S}} 
\newcommand\CFswtgp[2]{\CFsnoweight_{#1}(#2)}
\newcommand\CFs[1]{\CFswtgp k{#1}} 

\newcommand\MkKN{\MFswtgp k\KN}
\newcommand\MtwoKN{\MFswtgp2\KN}
\newcommand\SkKN{\CFswtgp k\KN}
\newcommand\StwoKN{\CFswtgp2\KN}
\newcommand\StwoKlevel[1]{\CFswtgp2{\paramodulargroup{#1}}}
\newcommand\SfourKN{\CFswtgp4\KN}

\newcommand\JkNcusp[2]{{\rm J}_{#1,#2}^{\rm cusp}}
\newcommand\JtwoNcusp{\JkNcusp2N}
\newcommand\JkmNcusp{\JkNcusp k{mN}}

\newcommand\JkNwh[2]{{\rm J}_{#1,#2}^{\rm w.h.}}
\newcommand\JzeroNwh{\JkNwh0N}

\newcommand\Grit{{\rm Grit}}
\newcommand\GritJkNcusp[2]{\Grit(\JkNcusp{#1}{#2})}
\newcommand\GritJtwoNcusp{\GritJkNcusp2N}

\newcommand{\projsp}{{\mathbb P}}

\newcommand\mymod{\text{ mod }}

\renewcommand{\Im}[1]{{\rm Im}(#1)}

\newcommand\iotast{\iota^*}
\newcommand\iotastf{\iota^*\!f}

\newcommand\TB{{\rm TB}}
\newcommand\BL{{\rm Borch}}

\newcommand\mmax{{m_{\rm max}}}

\theoremstyle{plain}
\newtheorem{theorem}{Theorem}[section]
\newtheorem{definition}[theorem]{Definition}
\newtheorem{lemma}[theorem]{Lemma}
\newtheorem{proposition}[theorem]{Proposition}
\newtheorem{corollary}[theorem]{Corollary}

\begin{document}
\title[Siegel Paramodular Forms of Weight~$2$ and Squarefree Level]
{Siegel Paramodular Forms of Weight~$2$ and Squarefree Level}

\author[C.~Poor]{Cris Poor}
\address{Dept.~of Mathematics, Fordham University, Bronx, NY 10458 USA}
\email{poor@fordham.edu}

\author[J.~Shurman]{Jerry Shurman}
\address{Reed College, Portland, OR 97202 USA}
\email{jerry@reed.edu}

\author[D.~Yuen]{David S.~Yuen}
\address{Dept.~of Mathematics and Computer Science, Lake Forest
College, 555 N.~Sheridan Rd., Lake Forest, IL 60045 USA}
\email{yuen@lakeforest.edu}

\date{\today}

\begin{abstract}
We compute the space $\StwoKN$ of weight~$2$ Siegel paramodular
cusp forms of squarefree level~$N<300$.  In conformance with the
paramodular conjecture of A.~Brumer and K.~Kramer, the space is only
the additive (Gritsenko) lift space of the Jacobi cusp form space
$\JtwoNcusp$ except for $N=249,295$, when it further contains one
nonlift newform.
For these two values of~$N$, the Hasse--Weil $p$-Euler factors of a
relevant abelian surface match the spin $p$-Euler factors of the
nonlift newform for the first two primes~$p\nmid N$.
\end{abstract}


\maketitle

\section{Introduction\label{secI}}

The paramodular conjecture of A.~Brumer and K.~Kramer \cite{brkr14}
says, in slight paraphrase and restricted to the case of abelian surfaces:
\begin{quote}
{\em For any positive integer~$N$ there is a one-to-one correspondence
  between isogeny classes of abelian surfaces~$A$ over~$\Q$ of
  conductor~$N$ with ${\rm End}_\Q(A)=\Z$, and lines $\C f$ for
  nonlift degree~$2$ Siegel paramodular Hecke newforms $f$ of
  weight~$2$ and level~$N$ having rational eigenvalues.  Moreover,
  the Hasse--Weil $L$-function of~$A$ and the spin $L$-function of~$f$
  should match, and the $\ell$-adic representation of
  $\mathbb T_\ell(A)\otimes\Q_\ell$ should be isomorphic to those
  associated to~$f$ for any~$\ell$ prime to~$N$.}
\end{quote}
Here the lift space is $\GritJtwoNcusp$, the Gritsenko (additive) lift
of the Jacobi cusp form space of weight~$2$ and index~$N$.  This lift
space lies inside the Siegel paramodular cusp form space
$\StwoKN$---the subscript~$2$ indicates the weight, and $\KN$ denotes
the paramodular group of degree~$2$ and level~$N$; the degree is
omitted from the notation because all Siegel modular forms in this
article have degree~$2$.
Newforms on~$\KN$ are by definition Hecke eigenforms orthogonal to the
images of level-raising operators from paramodular forms of lower levels
\cite{robertsschmidt07}.
Notation and terminology will be reviewed in section~\ref{secB}.

In \cite{py15}, the first and third authors of this article studied
$\StwoKN$ for prime levels $N<600$.  Because $\dim\StwoKN$ is unknown
in general, algorithms were used to bound this dimension by working
in~$\SfourKN$, whose dimension is known for prime~$N$ by work of
T.~Ibukiyama \cite{i85,i07}.  The algorithms proved that
$\StwoKN=\GritJtwoNcusp$ for all prime~$N<600$ other than the
exceptional cases $N=277,349,353,389,461,523,587$, precisely the
prime~$N<600$ for which relevant abelian surfaces exist.  Also,
$\StwoKlevel{277}$ contains one nonlift dimension $\C f_{277}$, and
A.~Brumer and J.~Voight and the first and third authors of this
article have shown that the equality of $L(f_{277},s,{\rm spin})$ and
$L(A_{277},s,\text{Hasse--Weil})$ holds conditionally on the existence
of certain Galois representations \cite{yuen15}.  A nonlift eigenform
in $\StwoKlevel{587}^-$ has been constructed as well \cite{gpy16}.  We
are currently working on constructing nonlift forms in the remaining
levels.

In \cite{bpy16}, J.~Breeding and the first and third authors of this
article showed that $\StwoKN=\GritJtwoNcusp$ for all~$N\le60$.
A key method here was Jacobi restriction, to be described briefly in
section~\ref{secB}, and see also \cite{ipy13}.  The article \cite{bpy16}
established a sufficient number of Fourier--Jacobi coefficients
necessary to make Jacobi restriction rigorous at a given level.
Running Jacobi restriction to this many coefficients was tractable for
levels up to~$60$.  One idea of the present article is that new
algorithms reduce the number of Fourier--Jacobi coefficients known to
be sufficient to certify Jacobi restriction, making the number small
enough that running Jacobi restriction with that many coefficients is
tractable for higher levels.

This article reports our investigation for squarefree composite
levels~$N<300$, necessarily using different methods from \cite{py15}.
Among these levels that are also odd, isogeny classes of abelian
surfaces exist only for $N=249$ and $N=295$, and at those two levels,
the one known isogeny class contains Jacobians of hyperelliptic curves.
Specifically, we may take $A_{249}$ to be the Jacobian of
$y^2=x^6+4x^5+4x^4+2x^3+1$ and $A_{295}$ to be the Jacobian of
$y^2=x^6-2x^3-4x^2+1$ \cite{brkr14}.  Our computations affirm that
indeed $\StwoKN$ contains one nonlift dimension for $N=249,295$, but
otherwise is only the lift space $\GritJtwoNcusp$.  For $N=249,295$,
we construct the nonlift eigenform as the sum of a Borcherds product
and a Gritsenko lift, and we show that its first two good spin
$p$-Euler factors match the Hasse--Weil $p$-Euler factors of the
relevant abelian surface.  Thus the main result of this article is
as follows.

\begin{theorem}
For all composite squarefree $N<300$ except $N=249$ and $N=295$,
$\StwoKN=\GritJtwoNcusp$.  For $N=249$ and $N=295$, $\StwoKN$ contains
one nonlift newform dimension $\C f_N$ beyond $\GritJtwoNcusp$.
For the abelian surfaces $A_{249}$ and~$A_{295}$ defined in the
previous paragraph, the $p$-Euler factor of
$L(f_{249},s,\text{spin\/})$ matches the $p$-Euler factor of
$L(A_{249},s,\text{Hasse--Weil\/})$ for $p=2$ and~$p=5$, and the
$p$-Euler factor of $L(f_{295},s,\text{spin\/})$ matches the $p$-Euler
factor of $L(A_{295},s,\text{Hasse--Weil\/})$ for $p=2$ and~$p=3$.
\end{theorem}

For the level~$249$ nonlift newform, let $\vartheta(\tau,z)$ be
Jacobi's odd theta function and let $\vartheta_r(\tau,z)=\vartheta(\tau,rz)$
for~$r\in\Zpos$, and introduce a product of quotients of theta functions,
$$
\psi_{249}=\frac{\vartheta_8}{\vartheta_1}
\frac{\vartheta_{18}}{\vartheta_6}\frac{\vartheta_{14}}{\vartheta_7}\,.
$$
The nonlift newform of~$\StwoKlevel{249}$ is then
\begin{align*}
f_{249}&=
14\,\BL(\psi_{249})\\
&\quad-6\,\Grit(\TB(2; 2,3,3,4,5,6,7,9,10,13))\\
&\quad-3\,\Grit(\TB(2;2,2,3,5,5,6,7,9,11,12 ))\\
&\quad+3\,\Grit(\TB(2;1,3,3,5,6,6,6,9,11,12 ))\\
&\quad+2\,\Grit(\TB(2; 1,1,2,3,4,5,6,9,10,15))\\
&\quad+7\,\Grit(\TB(2; 1,2,3,3,4,5,6,9,11,14)),
\end{align*}
where ``$\BL$'' and ``$\Grit$'' and~``$\TB$'' respectively denote
the Borcherds product, the Gritsenko lift, and the theta block
construction.
This will be explained further in section~\ref{secL249295}.
Thus $f_{249}$ is congruent to a Gritsenko lift modulo~$14$, and we
note that the two isogenous Jacobians of genus~$2$ curves of
conductor~$249$ defined over~$\Q$ have torsion groups $\Z/14\Z$
and~$\Z/28\Z$, as shown at {\tt lmfdb.org}.

Again with $\dim\StwoKN$ unavailable, we used algorithms to bound this
dimension by working in~$\SfourKN$, whose dimension is known by work
of T.~Ibukiyama and H.~Kitayama \cite{ik}.
These algorithms require spanning most of the weight~$4$ space but not
necessarily all of it: the requisite spanned subspace can fall short
essentially by the dimension of the Jacobi cusp form space $\JtwoNcusp$,
which is known by work of N.~Skoruppa and D.~Zagier \cite{ez85,sz89}.
Thus our major computational challenge was to span enough of~$\SfourKN$.
This space presents various impediments to diverse spanning methods,
so we had to employ a range of approaches.  The methods that worked
for prime~$N$ generally did not help for composite squarefree~$N$:
Hecke spreading, which spanned the Fricke plus space for large prime
levels, is obstructed for composite levels by the various
Atkin--Lehner signatures that are possible; and theta tracing, which
spanned the Fricke plus space for small prime levels and the Fricke
minus space for large prime levels, appears to be more expensive
computationally for composite levels.
For composite squarefree~$N$, our methods are to trace $(\GritJkNcusp2{Nq})^2$
and $\GritJkNcusp4{Nq}$ down to level~$N$ from level~$Nq$ for a small
prime~$q$ that does not divide~$N$, to Hecke spread $(\GritJtwoNcusp)^2$
at level~$N$, and to compute Borcherds products in the Fricke plus and
minus spaces at level~$N$.  When enough of the weight~$4$ space is
spanned, the algorithms for weight~$2$ show that Jacobi restriction
computations with only a small number of Fourier--Jacobi coefficients
are rigorous, and these computations give the results.

See the website \cite{yuen16b} for reports on the computations that
this article discusses.  For example, we sketch the online report for
level $N=286$, which the reader could examine alongside this paragraph.
The space $\CFswtgp4{\paramodulargroup{286}}$ has $189$ dimensions, of
which the lift space $\GritJkNcusp4{286}$ comprises~$48$.  Jacobi
restriction heuristically finds $113$ more Fricke plus space
dimensions, giving $161$ Fricke plus space dimensions altogether, and
$28$ Fricke minus space dimensions.  This heuristic information is
essential for targeting our constructions, \eg, deciding when to
switch from one method to another.  Tracing the weight~$4$ Gritsenko
lifts and the twofold products of weight~$2$ Gritsenko lifts from
level $1430=286\cdot5$ down to level $286$ and then adding the twofold
products of weight~$2$ Gritsenko lifts at level~$286$ gives $157$ plus
space dimensions and no minus space dimensions.  Hecke spreading gives
$8$ minus dimensions.  Adding in Borcherds products raises the spanned
plus and minus space dimensions to~$161$ and~$27$, so one dimension is
missing and we think that it lies in the minus space.  This gives
enough of the weight~$4$ space to run our weight~$2$ diagnostic tests,
to be described in section~\ref{secAW2W4}.
The $H_4(286,3,1)^+$ test says that weight~$2$ Jacobi restriction to
two or more terms gives a dimension upper bound of the Fricke plus
space $\StwoKlevel{286}^+$.  Thus, weight~$2$ Jacobi restriction to
five terms, which we have carried out, correctly bounds the dimension
of the Fricke plus space by~$3$, the dimension of the lift space
$\GritJkNcusp2{286}$; this shows that these two spaces are equal.  The
$H_4(286,1,1)^-$ test says that the Fricke minus space is~$0$.  So
altogether $\StwoKlevel{286}$ is the lift space.  Whereas the
$H_4(286,3,1)^+$ test says that Jacobi restriction to two or more
terms gives a Fricke plus space dimension upper bound, the theoretical
bound used in \cite{bpy16} says this only for $24$ or more terms.
This improvement is crucial: running Jacobi restriction systematically
across many levels to the number of terms required by the bound used
in \cite{bpy16} is computationally unviable for now.

The Borcherds products used in the computation at level~$N=286$
are given at the website \cite{yuen16b}.
For the Fricke plus space, the relevant file at the website explains
that $60$ Borcherds products $\BL(\psi)$ were used to find the
additional four dimensions reported; each $\psi$ lies in the space
$\JkNwh0{286}$ of weight~$0$, index~$286$ weakly holomorphic Jacobi
forms, and is a linear combination of a basis of
$\JkNcusp{12}{286}/\Delta_{12}$ where $\Delta_{12}\in\JkNcusp{12}0$ is
the classical discriminant function.
The basis of $\JkNcusp{12}{286}/\Delta_{12}$, built from theta blocks,
is given at the end of the file.  The website's file for the Fricke
minus space is similar, explaining that $22$ Borcherds products
$\BL(\psi)$ were used to find the additional $19$ dimensions reported,
with each $\psi$ now the sum of a quotient $\phi|V_2/\phi$ and a
linear combination of the basis of $\JkNcusp{12}{286}/\Delta_{12}$;
here $\phi$ is a theta block and $V_2$ is an index-raising Hecke
operator.  Our Borcherds product files will be described further in
section~\ref{secBP}.

Section~\ref{secB} gives background for this article.
Section~\ref{secC} shows that for low or odd weight and squarefree
level, either all Siegel paramodular forms are cusp forms or the vanishing
of a Siegel paramodular form's constant term suffices to make it a cusp form.
Section~\ref{secAW2W4} establishes the algorithms that study weight~$2$
Siegel paramodular cusp forms by working in weight~$4$.
Section~\ref{secTD} describes our tracing down method,
and section~\ref{secHS} describes our use of Hecke spreading.
Section~\ref{secBP} gives a result that certain conditions suffice for
a Borcherds product of low weight and squarefree level to be a Siegel
paramodular cusp form.
Finally, section~\ref{secL249295} describes how we used this result to
construct the weight~$2$ nonlift newforms at levels $N=249$ and~$N=295$.

We thank Fordham University for letting us carry out computations on
its servers.  We thank Reed College for making its computer lab
machines available to us, and especially the second author thanks
B.~Salzberg for helping him use them in parallel.

\section{Background\label{secB}}

We introduce notation and terminology for Siegel paramodular forms.
The degree~$2$ symplectic group $\Spgp(2)$ of $4\times4$ matrices is
defined by the condition $g'Jg=J$, where the prime denotes matrix
transpose and $J$ is the skew form $\smallmat0{-1}1{\phantom{-}0}$
with each block $2\times2$.
The Klingen and Siegel parabolic subgroups of~$\Spgp(2)$ are respectively
$$
\Ptwoone=\lset
\left[\begin{array}{cc|cc}
* & 0 & * & * \\ * & * & * & * \\
\hline
* & 0 & * & * \\ 0 & 0 & 0 & *
\end{array}\right]\rset,
\qquad
\Ptwozero=\lset
\left[\begin{array}{cc|cc}
* & * & * & * \\ * & * & * & * \\
\hline
0 & 0 & * & * \\ 0 & 0 & * & *
\end{array}\right]\rset.
$$
For $\Ptwoone$, the three zeros on the bottom row force the other two
zeros in consequence of the matrices being symplectic.
For any positive integer~$N$, the paramodular group $\KN$ of degree~$2$
and level~$N$ is the group of rational symplectic matrices that
stabilize the column vector lattice $\Z\oplus\Z\oplus\Z\oplus N\Z$.  In
coordinates,
$$
\KN=\lset
\left[\begin{array}{cc|cc}
* & *N & * & * \\ * & * & * & */N \\
\hline
* & *N & * & * \\ *N & *N & *N & *
\end{array}\right]\in\SptwoQ:\text{all $*$ entries integral}
\rset.
$$
Here the upper right entries of the four subblocks are ``more integral
by a factor of~$N$'' than implied immediately by the definition of the
paramodular group as a lattice stabilizer, but as with $\Ptwoone$ the
extra conditions hold because the matrices are symplectic.

Let $\UHPtwo$ denote the Siegel upper half space of $2\times2$
symmetric complex matrices that have positive definite imaginary part.
Elements of this space are notated
$$
\Omega=\cmat\tau zz\omega\in\UHPtwo,
$$
and also, letting $\e(z)=e^{2\pi iz}$ for~$z\in\C$, the notation
$$
q=\e(\tau),\quad\zeta=\e(z),\quad\xi=\e(\omega)
$$
is standard.
The real symplectic group $\SptwoR$ acts on~$\UHPtwo$ via fractional
linear transformations, $g(\Omega)=(a\Omega+b)(c\Omega+d)\inv$ for
$g=\smallmatabcd$, and the factor of automorphy is
$j(g,\Omega)=\det(c\Omega+d)$.  Fix an integer~$k$.
Any function $f:\UHPtwo\lra\C$ and any real symplectic matrix
$g\in\SptwoR$ combine to form another such function through the
weight~$k$ operator, $f\wtk g(\Omega)=j(g,\Omega)^{-k}f(g(\Omega))$.
A Siegel paramodular form of weight~$k$ and level~$N$ is a holomorphic
function $f:\UHPtwo\lra\C$ that is $\wtk{\KN}$-invariant; the K\"ocher
Principle says that for any positive $2\times2$ real matrix~$Y_o$, the
function $f\wtk g$ is bounded on $\lset\Im \Omega>Y_o\rset$ for all
$g\in\SptwoQ$.  The space of weight~$k$, level~$N$ Siegel paramodular
forms is denoted~$\MkKN$.
Dimension formula methods for $\MkKN$ based on the Riemann--Roch
Theorem or trace formulas are not available for~$k=2$.

The Witt map~$\iotast$ takes functions $f:\UHPtwo\lra\C$ to functions
$\iotastf:\UHP\times\UHP\lra\C$, with
$(\iotastf)(\tau\times\omega)=f(\smallmat\tau00\omega)$.
Especially, the Witt map takes $\MkKN$ to
$\MFs\SLtwoZ\otimes\MFs\SLtwoZ\wtk{\smallmat N001}$.
Siegel's $\Phi$ map takes any holomorphic function that has a Fourier
series of the form
$f(\Omega)=\sum_t\fc tf\,\e(\ip t\Omega)$, summing over matrices
$t=\smallmat nrrm$ with $n,m\in\Qnn$, $r\in\Q$, and $nm-r^2\ge0$, to
the function $(\Phi f)(\tau)=\lim_{\omega\to i\infty}(\iotastf)(\tau,\omega)$.
A Siegel paramodular form $f$ in~$\MkKN$ is a cusp form if $\Phi(f\wtk g)=0$
for all~$g\in\SptwoQ$; the space of such forms is denoted~$\SkKN$.
The dimension of $\SkKN$ is known for squarefree~$N$ and~$k\ge3$
\cite{i85,i07,ik}.

Every Siegel paramodular form of weight~$k$ and level~$N$ has a
Fourier expansion
$$
f(\Omega)=\sum_{t\in\XtwoNsemi}\fc tf\,\e(\ip t\Omega)
$$
where
$\XtwoNsemi=\lset\smallmat n{r/2}{r/2}{mN}:n,m\in\Znn,r\in\Z,4nmN-r^2\ge0\rset$
and $\ip t\Omega=\tr(t\Omega)$.  A Siegel paramodular form is a cusp
form if and only if its Fourier expansion is supported on $\XtwoN$,
defined by the strict inequality $4nmN-r^2>0$; this description of
cusp forms does not hold in general for groups commensurable
with~$\SptwoZ$, but it does hold for~$\KN$.
Consider any $\SptwoR$ matrix of the form $g=d^*\matoplus d=\smallmat{d^*}00d$
with $d\in\GLtwoR$, where the superscript asterisk denotes matrix
inverse-transpose.  Introduce the notation $t[u]=u'tu$ for compatibly
sized matrices $t$ and~$u$.  Then we have $f\wtk g(\Omega)
=(\det d)^{-k}\sum_{t\in\XtwoNsemi[d^*]}\fc{t[d']}f\,\e(\ip t\Omega)$
for any Siegel paramodular form~$f$, and especially if $g$
normalizes~$\KN$ so that $f\wtk g$ is again a Siegel paramodular form
then $\fc t{f\wtk g}=(\det d)^{-k}\fc{t[d']}f$ for~$t\in\XtwoNsemi$.
Let $\GampmsupzeroN$ denote the subgroup of~$\GLtwoZ$ defined by the
condition $b=0\mymod N$.  For $d\in\GampmsupzeroN$, the matrix
$g=d\inv\matoplus d'$ lies in~$\KN$ and we get
$\fc{t[d]}f=(\det d)^k\fc tf$ for~$t\in\XtwoNsemi$.
Our programs handle Fourier coefficient indices at the level
of $\GampmsupzeroN$-equivalence classes, each class having a
canonical representative of the form $t_o\times v_o$ with
$t_o\in\Xtwosemi$ Legendre reduced and with $v_o\in\Z/N\Z\oplus\Z/N\Z$.


For each positive divisor~$c$ of~$N$ such that $\gcd(c,N/c)=1$, let
$\hat c$ be a multiplicative inverse of~$N/c$ modulo~$c$.  Introduce
the $c$-th cusp representative matrix~$r_c$ and a translation
matrix~$\beta_c$ and their product, the $c$-th Atkin--Lehner matrix
$\alpha_c=r_c\beta_c$,
$$
r_c=\cmat10{N/c}1,
\quad
\beta_c=\frac1{\sqrt c}\cmat c{-\hat c}0{\phantom{-}1},
\quad
\alpha_c=\frac1{\sqrt c}\cmat c{-\hat c}N{1-(N/c)\hat c}.
$$
The Atkin--Lehner matrix normalizes the level~$N$ Hecke subgroup $\GzN$
of~$\SLtwoZ$, and it squares into~$\GzN$; these properties hold for
all of $\GzN\alpha_c\GzN$, any of which may be taken as~$\alpha_c$ instead.
For $c=1$ we take $r_1=\beta_1=\alpha_1=1_2$.
For $c=N$ we modify the Atkin--Lehner matrix, multiplying it from the
right by~$\smallmat10N1$ to get the traditional Fricke involution,
$\alpha_N=\tfrac1{\sqrt N}\smallmat0{-1}N{\phantom{-}0}:
\tau\mapsto-\tfrac1{N\tau}$.
Each Atkin--Lehner matrix $\alpha_c$ has a corresponding paramodular
Atkin--Lehner matrix $\mu_c=\alpha_c^*\matoplus\alpha_c$ that
normalizes~$\KN$ and squares into~$\KN$.
For $c=1$ we take $\mu_1=1_4$.
For $c=N$, the paramodular Fricke involution is
$\mu_N=\tfrac1{\sqrt N}(\smallmat0{-N}10\matoplus\smallmat0{-1}N0):
\smallpvar\longmapsto\smallmat{\omega N}{-z}{-z}{\tau/N}$.
The space $\CFs\KN$ decomposes as the direct sum of the Fricke eigenspaces
for the two eigenvalues~$\pm1$, $\SkKN=\SkKN^+\oplus\SkKN^-$.
More generally the Atkin--Lehner involutions satisfy
$\wtk{\alpha_c}\wtk{\alpha_{\tilde c}}=\wtk{\alpha_{c\tilde c}}$ for
coprime divisors $c$ and~$\tilde c$ of~$N$, and so they commute.
Thus $\SkKN$ decomposes as a direct sum of spaces $\SkKN^v$ where
$v$ is a vector of $\pm$ entries indexed by the prime divisors of the
level~$N$.  Such a vector is called an Atkin--Lehner signature.

The Fourier--Jacobi expansion of a Siegel paramodular cusp form
$f\in\SkKN$ is
$$
f(\Omega)=\sum_{m\ge1}\phi_m(f)(\tau,z)\xi^{mN},
\quad \Omega=\pvar,\ \xi=\e(\omega)
$$
with Fourier--Jacobi coefficients
$$
\phi_m(f)(\tau,z)
=\sum_{t=\smallmat n{r/2}{r/2}{mN}\in\XtwoN}\fc tfq^n\zeta^r,
\quad q=\e(\tau),\ \zeta=\e(z).
$$
Here the coefficient $\fc tf$ is also written $\fcJ{n,r}{\phi_m}$.
Each Fourier--Jacobi coefficient $\phi_m(f)$ lies in the space
$\JkNcusp k{mN}$ of weight~$k$, index~$mN$ Jacobi cusp forms,
whose dimension is known (for the theory of Jacobi forms, see
\cite{ez85,gn98,sz89}).
These are Jacobi forms of level one---this is an advantage of the
paramodular group over the Hecke subgroup $\Gamma_0^{(2)}(N)$
of~$\SptwoZ$---and trivial character, both omitted from the notation.
The additive (Gritsenko) lift $\Grit:\JkNcusp kN\lra\SkKN^\epsilon\subset\SkKN$
for $\epsilon=(-1)^k$ is a section of the map $\SkKN\lra\JkNcusp kN$
that takes each~$f$ to~$\phi_1(f)$, \ie, $\phi_1({\rm Grit}(\phi))=\phi$
for all $\phi\in\JkNcusp kN$.

Jacobi restriction is described briefly in section~5 of \cite{bpy16},
and we sketch it here as well.  Taking an even weight~$k$ for
simplicity, the coefficients of a Siegel paramodular Fricke eigenform
$f\in\SkKN^\epsilon$ satisfy the Siegel consistency relations
$\fcJ{n,r}{\phi_m}=\fcJ{n',r'}{\phi_{m'}}$ for
$\smallmat{n'}{r'/2}{r'/2}{m'N}\in\smallmat n{r/2}{r/2}{mN}[\GampmsupzeroN]$,
and they satisfy the Fricke eigenform relations
$\fcJ{n,r}{\phi_m}=\epsilon\fcJ{m,-r}{\phi_n}$.
For a chosen value $\mmax$, define a subspace $V(\mmax)$ of
$\bigoplus_{m=1}^\mmax\JkmNcusp$ by imposing any subset of the just-mentioned
linear relations on the coefficients $\fcJ{n,r}{\phi_m}$ with $m\le\mmax$.
Thus $V(\mmax)$ contains the image of the map
$\SkKN^\epsilon\lra\bigoplus_{m=1}^{m_{\rm max}}\JkNcusp k{mN}$ that
takes each~$f$ to $(\phi_1(f),\dotsc,\phi_{m_{\rm max}}(f))$.
In particular, $\dim \SkKN^\epsilon\le\dim V(\mmax)$
for $\mmax$ large enough to make the map inject; theoretical estimates
for $\mmax$ in \cite{bpy16} guarantee injectivity, but they can be too
big for practical use.
For values of $k$ and improved values of~$\mmax$ relevant to this
article, we can span the spaces $\JkmNcusp$ for $m\le\mmax$ with theta
blocks, and so we can compute $\dim V(\mmax)$.  
For any prime~$q$, if we have bases of the Jacobi form spaces over
the field $\F_q$ of $q$ elements then the same computation modulo~$q$
gives the bound $\dim\SkKN^\epsilon\le\dim_{\F_q}V(\mmax)$.
Jacobi restriction is remarkably tractable, and it often gives
optimal dimension upper bounds for values of $\mmax$ much smaller than
the theoretical estimates.
Even when we don't know that $\mmax$ is large enough to guarantee the
bounds given by Jacobi restriction, those bounds are still very useful
{\em heuristic} estimates.  For example, our weight~$4$ computations
that made up the bulk of the project being described in this article
were not viable until we used such estimates from Jacobi restriction to
decide how many Fourier coefficients the computations should track,
and also our work in weight~$4$ at a given level often involved a
confident decision between searching for more Fricke plus space 
dimensions or more Fricke minus space dimensions based on the
heuristic dimensions of the two eigenspaces.


The Dedekind eta function and the odd Jacobi theta function are
\begin{align*}
\eta(\tau)&=q^{1/24}\prod_{n\in\Zpos}(1-q^n),\\
\vartheta(\tau,z)
&=\sum_{n\in\Z}(-1)^nq^{(n+1/2)^2/2}\zeta^{n+1/2}.
\end{align*}
Let $\vartheta_r(\tau,z)=\vartheta(\tau,rz)$ for any $r\in\Zpos$.
Quotients $\vartheta_r/\eta$ are the basic ingredients of the theta
block ``without denominator'' (see \cite{gsz}) associated to any
finitely supported function $\varphi:\Znn\lra\Z$ with $\varphi(r)\ge0$
for~$r\ge1$,
$$
\TB(\varphi)(\tau,z)=\eta(\tau)^{\varphi(0)}
\prod_{r\in\Zpos}(\vartheta_r(\tau,z)/\eta(\tau))^{\varphi(r)}.
$$
Any such theta block transforms as a Jacobi form of
weight~$k=\tfrac12\varphi(0)$ and of index
$m=\tfrac12\sum_{r\in\Zpos}r^2\varphi(r)$, and when
$\tfrac1{24}\varphi(0)+\tfrac1{12}\sum_{r\in\Zpos}\varphi(r)\in\Z$ 
it has trivial character.
The theta block $\TB(\varphi)$ needn't be a Jacobi cusp form, but the
``without denominator'' stipulation that $\varphi(r)\ge0$ for~$r\ge1$
makes it lie in the space $\JkNwh km$ of weakly holomorphic
weight~$k$, index~$m$ Jacobi forms, whose Fourier expansions
$\psi(\tau,z)=\sum_{n,r}c(n,r)q^n\zeta^r$
are supported on~$n\gg-\infty$.  We show that equivalently, the
support can be taken to be $4nm-r^2\gg-\infty$.
The index~$m$ Jacobi form transformation
law $\psi(\tau,\lambda\tau+z)q^{\lambda^2m}\zeta^{2\lambda m}=\psi(\tau,z)$
for any $\lambda\in\Z$ shows that
$c(n-\lambda r+\lambda^2m,r-2\lambda m)=c(n,r)$
for all $(n,r)$ and $\lambda$, and also
$4(n-\lambda r+\lambda^2m)m-(r-2\lambda m)^2=4nm-r^2$.
Thus, for a given value of $4nm-r^2$ we may consider only coefficients
$c(n,r)$ with $|r|\le m$.
If for some~$n_o$, all coefficients $c(n,r)$ where $n<n_0$ are~$0$,
then all coefficients $c(n,r)$ where $4nm-r^2<4n_om-m^2$ are~$0$;
indeed, we may take $|r|\le m$, giving $4nm-m^2\le4nm-r^2<4n_om-m^2$
and thus $n<n_o$, so $c(n,r)=0$ as claimed.
Conversely, if for some~$d_0$, all coefficients $c(n,r)$ where
$4nm-r^2<d_o$ are~$0$, then also $c(n,r)=0$ for all $n<d_o/(4m)$.
Thus the weight~$k$, index~$m$ weakly holomorphic Jacobi forms
can be defined by the condition $c(n,r)=0$ either for $n\gg-\infty$ or
for $4nm-r^2\gg-\infty$, as claimed.
Furthermore, given a weakly holomorphic Jacobi form of index~$m$,
its coefficients $c(n,r)$ where $4nm-r^2\le0$ are entirely determined
by the finitely many such coefficients indexed by $(n,r)$ such that
$n\le m/4$ and $|r|\le m$.  This holds because
$c(n,r)=c(\tilde n,\tilde r)$ for some $(\tilde n,\tilde r)$ with
$|\tilde r|\le m$ and $4\tilde nm-\tilde r^2=4nm-r^2$; thus
$4\tilde nm-m^2\le4\tilde nm-\tilde r^2=4nm-r^2\le0$, and the
claimed inequality $\tilde n\le m/4$ follows.

Some functions $\varphi:\Znn\lra\Z$ that don't take $\Zpos$ to~$\Znn$
still produce weakly holomorphic Jacobi forms under the formula in the
previous display.  These are theta blocks ``with denominator.''  Our
algorithm and program to find Borcherds products to help span spaces
$\SfourKN$ involved theta blocks without denominator, while our
construction of the nonlifts in $\StwoKN$ for $N=249,295$ used theta
blocks with denominator.

\section{Cuspidality for Low Or Odd Weight and Squarefree Level\label{secC}}

Our computation used the following cuspidality test.  Specifically,
the test will be used in the proof of Corollary~\ref{BPthmcor}, which
identifies some Borcherds products as paramodular cusp forms.

\begin{proposition}\label{cuspidalityprop}
Let $N$ be a squarefree positive integer, and let $k$ be a positive integer.
If $k=2$ or $k$ is odd then $\MtwoKN=\StwoKN$.
If $k=4,6,8,10,14$ then for all $f\in\MkKN$, $f\in\SkKN$ if and only
if $\fc0f=0$.
\end{proposition}

\begin{proof}
Recall the matrices $r_c$, $\alpha_c$, $\beta_c$,
and $\mu_c=\alpha_c^*\matoplus\alpha_c$ from section~\ref{secB}.
For any squarefree positive integer~$N$, every divisor~$c$ of~$N$
satisfies the Atkin--Lehner condition $\gcd(c,N/c)=1$, and so
H.~Reefschl\"ager's decomposition (\cite{reefschlager73}, and see
Theorem~1.2 of \cite{py13})
$\SptwoQ=\bigsqcup_{0<c\mid N}\KN(r_c^*\matoplus r_c)\Ptwoone(\Q)$
combines with the relations $\alpha_c=r_c\beta_c$ to give
$\SptwoQ=\bigsqcup_{0<c\mid N}\KN\mu_c(\beta_c'\matoplus\beta_c\inv)
\Ptwoone(\Q)$.
Thus any $g\in\SptwoQ$ is $g=\kappa\mu_cu_c$ where $\kappa\in\KN$
and $0<c\mid N$ and $u_c\in\Ptwoone(\Q(\sqrt{N/c}))$,
and consequently $\Phi(f\wtk g)=\Phi(f\wtk{\mu_c}\wtk{u_c})$.
Let $u_{1,c}\in\GLtwoR$ denote the $2\times2$ matrix of upper left
entries of the four blocks of~$u_c$.  For any $f\in\MFs\KN$, a
computation shows that
$\Phi(f\wtk{\mu_c}\wtk{u_c})=(m\,\Phi(f\wtk{\mu_c}))\wtk{u_{1,c}}$,
where $m$ is a nonzero constant that depends on~$u_c$.
Thus, to show that $f$ is a cusp form, it suffices to show that
$\Phi(f\wtk{\mu_c})=0$ for all $0<c\mid N$.
Note that to do so, we need to consider the Siegel $\Phi$ map only
on~$\MFs\KN$.  Recall that the Witt map on~$\MFs\KN$ has codomain
$\MFs\SLtwoZ\otimes\MFs\SLtwoZ\wtk{\smallmat N001}$.

If $k=2$ or $k$ is odd then 
$\MFs\SLtwoZ\otimes\MFs\SLtwoZ\wtk{\smallmat N001}=0$,
making the Witt map on~$\MFs\KN$ zero and hence making
the Siegel $\Phi$ map on~$\MFs\KN$ zero.  This proves the first statement.

If $k=4,6,8,10,14$ then
$\MFs\SLtwoZ\otimes\MFs\SLtwoZ\wtk{\smallmat N001}=\C\varphi$ with
$\varphi(\tau,\omega)=E_k(\tau)E_k\wtk{\smallmat N001}(\omega)$.
The Witt map image of any $f$ in $\MkKN$ has the same constant term as~$f$.
Thus, letting $c_o$ denote the constant term of~$E_k\wtk{\smallmat N001}$,
the Witt map is $f\mapsto(\fc0f/c_o)\varphi$
and Siegel's $\Phi$ map is $f\mapsto\fc0fE_k$.
To prove the nontrivial part of the second statement, we take any
$f\in\MkKN$ with $\fc0f=0$ and show that $\Phi(f\wtk{\mu_c})=0$
for~$0<c\mid N$, making $f$ a cusp form.  The Fourier coefficients of
$f\wtk{\mu_c}$ are $\fc t{f\wtk{\mu_c}}=\fc{t[\alpha_c']}f$, so in
particular $\fc0{f\wtk{\mu_c}}=\fc0f=0$.  Thus
$\Phi(f\wtk{\mu_c})=\fc0{f\wtk{\mu_c}}E_k=0$, as desired.
\end{proof}

\section{Analyzing Weight~$2$ Via Weight~$4$\label{secAW2W4}}

Let $N$ be a positive integer.
This section presents four tests to study $\StwoKN$ based on
computations in $\SfourKN$.  A main point is that the tests can
certify that the results of Jacobi restriction are rigorous even when
the restriction is carried out only to a few terms.
We begin by introducing subspaces of~$\StwoKN$ whose vanishing
connotes the correctness of Jacobi restriction.

\begin{definition}
For any~$d\in\Zpos$ define
\begin{align*}
\StwoKN(d)
&=\lset f\in\StwoKN:f(\Omega)=\sum_{m\ge d}\phi_m(f)(\tau,z)\xi^{mN}\rset.
\end{align*}
Also define
$$
\StwoKN^\epsilon(d)=\StwoKN(d)\cap\StwoKN^\epsilon,\quad\epsilon=\pm1.
$$
\end{definition}

We say that elements of a space~$\StwoKN(d)$ are {\em $d$-docked\/},
because their Fourier--Jacobi coefficients before $\phi_d$ vanish; in
particular, {\em $1$-docked\/} connotes no conditions and {\em $2$-docked\/}
means that $\phi_1(f)=0$.  Any~$f\in\StwoKN(d)$ has Fourier
coefficients $\fc tf=0$ for all $t=\smallmat n{r/2}{r/2}{mN}\in\XtwoN$
such that $m<d$, and so it has Fourier coefficients
$\fc{t[\gamma]}f=0$ for all such~$t$ and for all
$\gamma\in\GampmsupzeroN$ because its Fourier coefficients are
$\GampmsupzeroN$ class functions, as discussed in section~\ref{secB}.
If $\StwoKN(d)=0$, so that $\StwoKN^\pm(d)=0$,
then running Jacobi restriction out to $d-1$ terms or more produces
rigorous upper bounds of~$\dim\StwoKN^\pm$.
If $\StwoKN^+(d)=0$ then running Jacobi restriction out to $d-1$ terms
or more produces a rigorous upper bound of~$\dim\StwoKN^+$, and
similarly with ``$-$'' in place of~''$+$''.  Note, however, that the
conditions $\StwoKN^\pm(d)=0$ need not imply $\StwoKN(d)=0$ for~$d\ge2$.
Note also that $\StwoKN(d)\cap\GritJtwoNcusp=0$ for~$d\ge2$.
We make two observations to be used in the analysis.

\begin{lemma}\label{H4Nlemma1}
{\rm(a) }Let $f\in\StwoKN$ be nonzero, and let $\lset g_i:i\in I\rset$
be a basis of~$\GritJtwoNcusp$.  The set $\lset fg_i:i\in I\rset$ is
linearly independent in~$\SfourKN$, and if $f$ is a nonlift then the
set $\lset f^2,fg_i:i\in I\rset$ is linearly independent as well.

{\rm(b) }We have the equivalence
$$
\StwoKN^+=\GritJtwoNcusp\iff\StwoKN^+(2)=0,
$$
and the same equivalence holds with $\StwoKN$ in place of~$\StwoKN^+$.
\end{lemma}

\begin{proof}
(a) Because there is no nontrivial linear relation among
$\lset g_i:i\in I\rset$, there is no such relation among
$\lset fg_i:i\in I\rset$ either, because the graded ring of paramodular
forms has no zero divisors.  The same argument shows that if
$f$ is a nonlift then there is no nontrivial linear relation among
$\lset f^2,fg_i:i\in I\rset$.

(b) A nonzero element of $\StwoKN^+(2)$ is an element of~$\StwoKN^+$
but not of~$\GritJtwoNcusp$, and an element of~$\StwoKN^+$ that is not
an element of~$\GritJtwoNcusp$ can be translated by an element
of~$\GritJtwoNcusp$ to produce a nonzero element of $\StwoKN^+(2)$.
The argument for $\StwoKN$ is the same.
\end{proof}

To study $\StwoKN(d)$ and $\StwoKN^\epsilon(d)$ for $\epsilon=\pm1$,
we introduce subspaces of $\SfourKN$ and $\SfourKN^\epsilon$ that are
generated by products of $\StwoKN$-elements subject to docking and
Fricke eigenspace conditions.

\begin{definition}\label{H4Ndef}
For any $d\in\Zpos$, define the following subspaces of~$\SfourKN$.
\begin{align*}
H_4(N,d,d)^+&=
\left\langle f_1f_2:f_1,f_2\in\StwoKN^\epsilon(d)
\text{ for one of $\epsilon=\pm1$}\right\rangle,\\
H_4(N,d,1)^{\phantom{+}}
&=\left\langle f_1f_2:f_1\in\StwoKN(d),\ f_2\in\StwoKN\right\rangle,\\
H_4(N,d,1)^+&=
\left\langle
f_1f_2:f_1\in\StwoKN^\epsilon(d),\ f_2\in\StwoKN^\epsilon
\text{ for one of $\epsilon=\pm1$}\right\rangle,\\
H_4(N,d,1)^-&=
\left\langle f_1f_2:
f_1\in\StwoKN^\epsilon(d),\ f_2\in\StwoKN^{-\epsilon}
\text{ for one of $\epsilon=\pm1$}
\right\rangle.
\end{align*}
\end{definition}

Dimension bounds of the weight~$4$ spaces $H_4(N)$ combine with
Lemma~\ref{H4Nlemma1}(a) to give information about the docked
weight~$2$ spaces as follows.

\begin{lemma}\label{H4Nlemma2}
Let $d$ be a positive integer.
\begin{enumerate}
\item If $H_4(N,d,d)^+=0$ then $\StwoKN^\pm(d)=0$.
\item If $\dim H_4(N,d,1)<\dim\JtwoNcusp+1$ then
  $\StwoKN(d)\subset\GritJtwoNcusp$.
\item If $\dim H_4(N,d,1)^+<\dim\JtwoNcusp+1$ then
  $\StwoKN^+(d)\subset\GritJtwoNcusp$.
\item If $\dim H_4(N,d,1)^-<\dim\JtwoNcusp$ then $\StwoKN^-(d)=0$.
\end{enumerate}
\end{lemma}

\begin{proof}
(1) Every $f\in\StwoKN^\epsilon(d)$ for either of
$\epsilon=\pm1$ squares into~$H_4(N,d,d)^+$, which is~$0$, so
$\StwoKN^\pm(d)=0$.

(2) If some $f\in\StwoKN(d)$ is a nonlift then
the linear independence of $\lset f^2,fg_i:i\in I\rset$ gives $\dim
H_4(N,d,1)\ge\dim\JtwoNcusp+1$, and so the result follows by contraposition.

This same argument, but with plus spaces, gives~(3).

(4) If some $f\in\StwoKN^-(d)$ is nonzero then
the linear independence of $\lset fg_i:i\in I\rset$ gives $\dim
H_4(N,d,1)^-\ge\dim\JtwoNcusp$, and so the result follows by contraposition.
\end{proof}

\newcommand\ttS{{\tt S}}
\newcommand\ttd{\dim{\tt S}}
\newcommand\ttM{{\tt M}}
\newcommand\ttr{\operatorname{rank}{\tt M}}

A crucial idea is that we can establish computable dimension estimates
for the $H_4(N)$ spaces of Definition~\ref{H4Ndef}.  The Fourier
coefficient formula for the product of two weight~$2$ paramodular forms
$f_1,f_2\in\StwoKN$ is
$$
\fc t{f_1f_2}=\sum_{\substack{t_1,t_2\in\XtwoN\\t_1+t_2=t}}
\fc{t_1}{f_1}\fc{t_2}{f_2},\quad t\in\XtwoN.
$$
Recall that the Fourier coefficients of $f_1$
and~$f_2$ are $\GampmsupzeroN$ class functions.
For any $t\in\XtwoN$, let $m_N(t)=
\min\lset m:\smallmat n{r/2}{r/2}{mN}\in t[\GampmsupzeroN]\rset$.
(The ``$m$'' in this function's name stands for ``minimum function,''
although also the quantity being minimized is named~$m$.)
Computing $m_N(t)$ is a finite process, because for any
given~$m$ we may take $|r|\le mN$, then search for matrices
$\smallmat n{r/2}{r/2}{mN}$ such that $4nmN-r^2=4\det t$, and then
check whether each such matrix lies in~$t[\GampmsupzeroN]$.
Consider an index~$t\in\XtwoN$, and let $d$ be a positive integer.
Suppose that for any pair $t_1\times t_2$ of $\XtwoN$~matrices such
that $t_1+t_2=t$, necessarily $m_N(t_1)<d$ or $m_N(t_2)<d$.
This condition combines with the previous display to say that
$\fc t{f_1f_2}=0$ for every generating product $f_1f_2$ of~$H_4(N,d,d)^+$,
and consequently $\fc tf=0$ for all $f\in H_4(N,d,d)^+$.
Similarly, if for any pair $t_1\times t_2$ such that $t_1+t_2=t$, necessarily
$m_N(t_1)<d$, then $\fc tf=0$ for all $f\in H_4(N,d,1)$.

In the next proposition, typewriter font is used to denote variables
that we compute in practice.  We remind the reader that $\dim\SfourKN$
and $\dim\JtwoNcusp$ are known \cite{ik,ez85,sz89}, but not currently
$\dim\SfourKN^\pm$.

\begin{proposition}\label{dimH4Nboundsprop}
Let $N$ be a positive integer.
Let $\ttS^\pm$ be subspaces of $\SfourKN^\pm$,
and let $\ttS=\ttS^+\oplus\ttS^-\subset\SfourKN$.
For each~$s$ in~$\lset+,-,\text{empty character}\rset$,
let $\lset g_i^s\rset$ be a basis of~$\ttS^s$.
Let $d$ be a positive integer.
For each~$\delta$ in~$\lset d,1\rset$ and each~$s$ as above,
let $\SetT(\ttS^s)$ be a determining set of Fourier coefficient indices
for~$\ttS^s$, and let
$$
\SetT(\ttS^s,d,\delta)=\lset
t\in\SetT(\ttS^s):
\left(\begin{aligned}
&\text{if $t=t_1+t_2$ where $t_1,t_2\in\XtwoN$}\\
&\text{then }
m_N(t_1)<d\text{ or }m_N(t_2)<\delta
\end{aligned}\right)
\rset.
$$
{\rm(}For $\delta=1$, the condition $m_N(t_2)<\delta$ is
impossible, leaving a condition on~$t_1$.\/{\rm)}
Define the $\ttd^s\times|\SetT(\ttS^s,d,\delta)|$ matrix
$\ttM(\ttS^s,d,\delta)=[\fc t{g_i^s}]$.  We have the following bounds.
\begin{enumerate}
\item $\dim H_4(N,d,d)^+\le\dim\SfourKN-\ttd^--\ttr(\ttS^+,d,d)$,
\item $\dim H_4(N,d,1)^{\phantom{+}}\le\dim\SfourKN-\ttr(\ttS,d,1)$,
\item $\dim H_4(N,d,1)^+\le\dim\SfourKN-\ttd^--\ttr(\ttS^+,d,1)$,
\item $\dim H_4(N,d,1)^-\le\dim\SfourKN-\ttd^+-\ttr(\ttS^-,d,1)$.
\end{enumerate}
\end{proposition}

\begin{proof}
We prove the first bound.
An element of~$H_4(N,d,d)^+$ either lies outside $\ttS^+$ or it lies
in~$\ttS^+$ and its Fourier series expansion truncation is~$0$
on~$\SetT(N,d,d)$.  Thus $\dim H_4(N,d,d)^+$ is at most the sum of
$\dim\SfourKN^+-\ttd^+$ and the left nullity of $\ttM(\ttS^+,d,d)$.
But $\dim\SfourKN^+\le\dim\SfourKN-\ttd^-$ and the left nullity is
$\ttd^+-\ttr(\ttS^+,d,d)$, so the stated bound follows.
The other three bounds are established similarly.
%
%
%
\end{proof}

In practice we grow our spanned subspaces~$\ttS^\pm$ of~$\SfourKN^\pm$
until $\ttd=\dim\SfourKN$ if $\JtwoNcusp=0$, or
$\dim\SfourKN-\ttd<\dim\JtwoNcusp$ if $\JtwoNcusp\ne0$.
As $d$ grows, we expect $\ttr(\ttS^s,d,\delta)$ to grow toward~$\ttd^s$,
making all four bounds in Proposition~\ref{dimH4Nboundsprop} decrease
toward $\dim\SfourKN-\ttd$, and so we expect the conditions of
Lemma~\ref{H4Nlemma2} to apply for large enough~$d$.
Building on Lemma~\ref{H4Nlemma2} and Proposition~\ref{dimH4Nboundsprop},
we have diagnostic tests for the weight~$2$ spaces as follows.

\begin{proposition}
Let $\ttS^\pm$ and~$\ttS$ be as in Proposition~\ref{dimH4Nboundsprop}.
Let $d$ be a positive integer, and let
$\ttM(\ttS^+,d,d)$, $\ttM(\ttS,d,1)$, and~$\ttM(\ttS^\pm,d,1)$ be as in
Proposition~\ref{dimH4Nboundsprop}.
\begin{enumerate}
\item Suppose that $\dim\SfourKN=\ttd^-+\ttr(\ttS^+,d,d)$.
  If $d=1$ then $\StwoKN=0$;
  if $d=2$ then $\StwoKN^+=\GritJtwoNcusp$ and $\StwoKN^-(2)=0$;
  if $d\ge3$ then $\StwoKN^\pm(d)=0$.
\item Suppose that $\dim\SfourKN-\ttr(\ttS,d,1)<\dim\JtwoNcusp+1$.
  If $d=1,2$ then $\StwoKN=\GritJtwoNcusp$;
  if $d\ge3$ then $\StwoKN^\pm(d)=0$.
\item Suppose that $\dim\SfourKN-\ttd^--\ttr(\ttS^+,d,1)<\dim\JtwoNcusp+1$.
  If $d=1,2$ then $\StwoKN^+=\GritJtwoNcusp$;
  if $d\ge3$ then $\StwoKN^+(d)=0$.
\item Suppose that $\dim\SfourKN-\ttd^+-\ttr(\ttS^-,d,1)<\dim\JtwoNcusp$.
  If $d=1$ then $\StwoKN^-=0$.
  If $d\ge2$ then $\StwoKN^-(d)=0$.
\end{enumerate}
\end{proposition}

We reiterate that the conclusion $\StwoKN^-(2)=0$ in the first case
implies that Jacobi restriction to one or more terms produces a
rigorous upper bound of $\dim\StwoKN^-$, and similarly for the other
cases.

Before proving the proposition, we note that in the first case, the
condition can hold only if $\ttd=\dim\SfourKN$, so this equality
should be checked before computing~$\ttr(\ttS^+,d,d)$, and when this
equality does hold, the condition simplifies to $\ttr(\ttS^+,d,d)=\ttd^+$.
Similarly, in the second and third cases, the condition can hold
only if $\dim\SfourKN-\ttd<\dim\JtwoNcusp+1$, and this inequality
should be checked before computing $\ttr(\ttS,d,1)$ or~$\ttr(\ttS^+,d,1)$.
In the fourth case, the condition can hold only if
$\dim\SfourKN-\ttd<\dim\JtwoNcusp$, and this should be checked before
computing~$\ttr(\ttS^-,d,1)$.

\begin{proof}
For~(1), by Lemma~\ref{H4Nlemma2} and
Proposition~\ref{dimH4Nboundsprop}, $\StwoKN^\pm(d)=0$.  For $d=1$
these equalities give $\StwoKN=0$, and for $d=2$ the first equality
implies that $\StwoKN^+=\GritJtwoNcusp$ by Lemma~\ref{H4Nlemma1}(b).

For~(2), by Lemma~\ref{H4Nlemma2} and
Proposition~\ref{dimH4Nboundsprop}, $\StwoKN(d)\subset\GritJtwoNcusp$.
For $d=1$ this containment gives $\StwoKN=\GritJtwoNcusp$.  For $d=2$ it
says that $\StwoKN(2)=0$ because $\StwoKN(2)\cap\GritJtwoNcusp=0$, and
so $\StwoKN=\GritJtwoNcusp$ by Lemma~\ref{H4Nlemma1}(b).  For $d\ge3$
it says that $\StwoKN(d)=0$ because $\StwoKN(d)\cap\GritJtwoNcusp=0$.

The same argument, but with plus spaces, gives~(3).

For~(4), by Lemma~\ref{H4Nlemma2} and
Proposition~\ref{dimH4Nboundsprop}, $\StwoKN^-(d)=0$.
\end{proof}

We have run Jacobi restriction to five or more terms for all of the
spaces $\StwoKN^\pm$ where the level $N$ is composite and squarefree
in~$\lset62,\dotsc,299\rset$.  If one of the tests above applies for
some~$d\le6$ then the corresponding dimension upper estimate provided
by Jacobi restriction to five or more terms is rigorous.
The heuristic upper bounds provided by Jacobi restriction for the
just-mentioned levels~$N$ are
\begin{alignat*}2
&\dim\StwoKN^+\le\dim\JtwoNcusp&&\quad\text{for }N\ne249,295,\\
&\dim\StwoKN^+\le\dim\JtwoNcusp+1&&\quad\text{for }N=249,295,\\
&\dim\StwoKN^-=0&&\quad\text{for all }N.
\end{alignat*}
Thus our tests to certify that Jacobi restriction to five terms
gives rigorous upper bounds are as follows.

$\mathbf{H_4(N,d,d)^+}$ {\bf test}, can succeed only if $\ttd=\dim\SfourKN$:
For $d=1,\dotsc,6$, if $\ttr(\ttS^+,d,d)=\ttd^+$ then
\begin{alignat*}2
&\StwoKN=\GritJtwoNcusp
&&\quad\text{if $d=1$ or~$2$, or $d\ge3$ and $N\ne249,295$},\\
&\left(\begin{aligned}
&\dim\StwoKN^+\le\dim\JtwoNcusp+1\\
&\text{and}\ \StwoKN^-=0
\end{aligned}\right)
&&\quad\text{if $d\ge3$ and $N=249,295$}.
\end{alignat*}
For $d=1$ this test can conclude that $\StwoKN=0$, but the
given conclusion is all that we need.

$\mathbf{H_4(N,d,1)}$ {\bf test}, can succeed only if
$\dim\SfourKN-\ttd<\dim\JtwoNcusp+1$:
For $d=1,\dotsc,6$, if $\dim\SfourKN-\ttr(\ttS,d,1)<\dim\JtwoNcusp+1$ then
\begin{alignat*}2
&\StwoKN=\GritJtwoNcusp
&&\quad\text{if $d=1$ or~$2$, or $d\ge3$ and $N\ne249,295$},\\
&\left(\begin{aligned}
&\dim\StwoKN^+\le\dim\JtwoNcusp+1\\
&\text{and}\ \StwoKN^-=0
\end{aligned}\right)
&&\quad\text{if $d\ge3$ and $N=249,295$}.
\end{alignat*}

$\mathbf{H_4(N,d,1)^+}$ {\bf test}, can succeed only if
$\dim\SfourKN-\ttd<\dim\JtwoNcusp+1$: For $d=1,\dotsc,6$,
if $\dim\SfourKN-\ttd^--\ttr(\ttS^+,d,1)<\dim\JtwoNcusp+1$ then
\begin{alignat*}2
&\StwoKN^+=\GritJtwoNcusp
&&\quad\text{if $d=1$ or~$2$, or $d\ge3$ and $N\ne249,295$},\\
&\dim\StwoKN^+\le\dim\JtwoNcusp+1
&&\quad\text{if $d\ge3$ and $N=249,295$}.
\end{alignat*}

$\mathbf{H_4(N,d,1)^-}$ {\bf test}, can succeed only if
$\dim\SfourKN-\ttd<\dim\JtwoNcusp$: For $d=1,\dotsc,6$,
if $\dim\SfourKN-\ttd^+-\ttr(\ttS^-,d,1)<\dim\JtwoNcusp$ then
$$
\StwoKN^-=0.
$$

\section{Tracing Down\label{secTD}}

Let $N$ be a squarefree positive integer, and let $q$ be a prime that
does not divide~$N$.  We define and compute an averaging {\em trace down\/} 
operator
$$
{\rm TrDn}:\CFs\KNq\lra\CFs\KN.
$$
Here $\KNq$ is not a subgroup of~$\KN$, but we have the configuration
of groups (in which $\Gamma_0'(N)$ denotes $\paramodulargroup N\cap\SptwoZ$
and similarly for~$\Gamma_0'(Nq)$)
$$
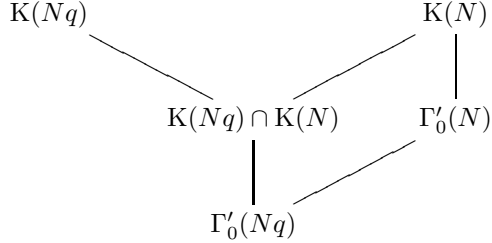
\xymatrix{
\KNq \ar@{-}[dr] && \KN \\
& \KNq\cap\KN \ar@{-}[ur] & \Gamma_0'(N) \ar@{-}[u] \\
& \Gamma_0'(Nq) \ar@{-}[u] \ar@{-}[ur]
}
$$
We will find representatives $\lset g_{1i}\rset$ of the quotient
space $\Gamma_0'(Nq)\bs\Gamma_0'(N)$ and representatives
$\lset g_{2j}\rset$ of the quotient space $\Gamma_0'(N)\bs\KN$, so
that altogether $\KN=\bigsqcup_{i,j}\Gamma_0'(Nq)g_{ij}$ with
$g_{ij}=g_{1i}g_{2j}$ for all~$i,j$.  Further, because $Nq$ is
squarefree, so that $\KNq$ has only one $0$-cusp, we can decompose
each coset representative $g=g_{ij}$ as $g=\kappa u$ with
$\kappa\in\KNq$ and $u\in\PtwozeroQ$.
Thus overall the quotient space $\Gamma_0'(Nq)\bs\KN$
is $\KN=\bigsqcup\Gamma_0'(Nq)\kappa u$, and because each~$\kappa$ lies
in~$\KNq$ the trace down operator is
$$
{\rm TrDn}\,f=\sum_u f\wtk u,\quad f\in\CFs\KNq.
$$
Some of the results in this section are well known, but we assemble
them here for the sake of a complete discussion in one place.

First we study $\Gamma_0'(Nq)\bs\Gamma_0'(N)$.
Let $\projsp^3(\Z/q\Z)=(\Z/q\Z)^4_{\rm prim}/(\Z/q\Z)^\times$
where $(\Z/q\Z)^4_{\rm prim}$ consists of all vectors
$(\ol a,\ol b,\ol c,\ol d)\in(\Z/q\Z)^4$ such that the ideal
of~$\Z/q\Z$ generated by the entries $\ol a,\ol b,\ol c,\ol d$ is all
of~$\Z/q\Z$; here the overbar denotes reduction of integers modulo~$q\Z$.
We show that $\projsp^3(\Z/q\Z)$ parametrizes $\Gamma_0'(Nq)\bs\Gamma_0'(N)$.

\begin{proposition}
Let $N$ be a positive integer and let $q$ be a prime that does not divide~$N$.
\begin{itemize}
\item[{\rm(1)}]Each element $\pi$ of~$\projsp^3(\Z/q\Z)$ has a representative
  $(a,b,c,d)\in\Z^4$ such that the vector $v_\pi=(aN,bN,cN,d)$ is
  primitive, \ie, $\gcd(aN,bN,cN,d)=1$.
\item[{\rm(2)}]Each such $v_\pi$ is the bottom row of a matrix~$g_\pi$
  in~$\Gamma_0'(N)$.
\item[{\rm(3)}]The map $\projsp^3(\Z/q\Z)\lra\Gamma_0'(Nq)\bs\Gamma_0'(N)$
  that takes $\pi$ to~$\Gamma_0'(Nq)g_\pi$ is well defined.  That is,
  the coset $\Gamma_0'(Nq)g_\pi$ depends only on~$\pi$, not on any
  choices made in constructing~$v_\pi$ from~$\pi$ or $g_\pi$ from~$v_\pi$.
\item[{\rm(4)}]The map is bijective.  That is, if $\pi$ and~$\pi'$ are
  distinct in~$\projsp^3(\Z/q\Z)$ then $\Gamma_0'(N)g_\pi$ and
  $\Gamma_0'(N)g_{\pi'}$ are distinct in~$\Gamma_0'(Nq)\bs\Gamma_0'(N)$,
  and cosets $\Gamma_0'(Nq)g_\pi$ constitute all
  of~$\Gamma_0'(Nq)\bs\Gamma_0'(N)$.
  Thus $[\Gamma_0'(N):\Gamma_0'(Nq)]=|\projsp^3(\Z/q\Z)|=1+q+q^2+q^3$.
\end{itemize}
\end{proposition}

\begin{proof}
(1) Consider any element
$\pi=(\ol\alpha,\ol\beta,\ol\gamma,\ol\delta)(\Z/q\Z)^\times$
of~$\projsp^3(\Z/q\Z)$.  Take a representative
$(\alpha,\beta,\gamma,\delta+mq)$ such that $\gcd(\delta+mq,N)=1$, and
then divide through by the greatest common divisor of the entries;
because $q$ cannot divide all four entries, this has no effect on the
element~$\pi$ represented.  We now have the desired representative
$(a,b,c,d)$ of~$\pi$ such that $\gcd(aN,bN,cN,d)=1$.

(2) Because the first three entries of~$v_\pi$ are multiples of~$N$,
and because any $\SptwoZ$ matrix with such a bottom row lies
in~$\Gamma_0'(N)$, it suffices to review the standard fact that any
primitive vector $r_4\in\Z^4$ is the bottom row of a matrix in~$\SptwoZ$.
There exists $r_2\in\Z^4$ such that $r_2Jr_4'=-1$.  Take a primitive
$r_3$ such that $r_3Jr_2'=r_3Jr_4'=0$, and then take $\rho_1$ such
that $\rho_1Jr_3'=-1$.  Let $r_1=\rho_1+(\rho_1Jr_4')r_2-(\rho_1Jr_2')r_4$.
Thus $r_1Jr_3'=-1$ and $r_1Jr_2'=r_1Jr_4'=0$.
The matrix with rows $r_1$ through~$r_4$ lies in~$\SptwoZ$.

(3) Consider any two $\Gamma_0'(N)$ matrices $g_\pi$ and~$\tilde g_\pi$
arising from the same element~$\pi$ of~$\projsp^3(\Z/q\Z)$.
We want to show that $\tilde g_\pi g_\pi\inv$ lies in~$\Gamma_0'(Nq)$.
Because $\Gamma_0'(N)$ is a group, all that needs to be shown is that
the first three entries of the bottom row of $\tilde g_\pi g_\pi\inv$
are multiples of~$q$.  The matrices $g_\pi$ and~$\tilde g_\pi$ have
bottom rows $r_4=(\alpha N,\beta N,\gamma N,\delta)$ and
$\tilde r_4=(\tilde\alpha N,\tilde\beta N,\tilde\gamma N,\tilde\delta)$
with $(\tilde\alpha,\tilde\beta,\tilde\gamma,\tilde\delta)
=\lambda(\alpha,\beta,\gamma,\delta)\mymod q$ for some $\lambda$
coprime to~$q$, and so working modulo~$q$ we may replace the bottom
row~$\tilde r_4$ of~$\tilde g_\pi$ by~$\lambda r_4$, a scalar multiple
of the bottom row of~$g_\pi$.  Thus the bottom row of $\tilde g_\pi g_\pi\inv$
is a multiple modulo~$q$ of the bottom row $(0,0,0,1)$
of~$g_\pi g_\pi\inv=1_4$, and we are done.

(4) For injectivity, consider any $\pi\in\projsp^3(\Z/q\Z)$ and
consider a matrix $g_\pi\in\Gamma_0'(N)$, whose bottom row is
$(\alpha N,\beta N,\gamma N,\delta)$ where
$\pi=(\ol\alpha,\ol\beta,\ol\gamma,\ol\delta)(\Z/q\Z)^\times$.
Any $h\in\Gamma_0'(Nq)$ has bottom row $(aNq,bNq,cNq,d)$, with $q\nmid d$,
and so the product $hg_\pi$ has bottom row
$(d\alpha N+*Nq,d\beta N+*Nq,d\gamma N+*Nq,d\delta+*Nq)$,
which is $((d\alpha+*q)N,(d\beta+*q)N,(d\gamma+*q)N,d\delta+*q)$.
This shows that the coset $\Gamma_0'(Nq)g_\pi$ consists entirely of
matrices~$g_{\pi'}$.
For surjectivity, any element~$g$ of~$\Gamma_0'(N)$ has bottom row
$(aN,bN,cN,d)$ with the vector primitive.  Thus $(a,b,c,d)$ is
primitive as well, and so it represents an element~$\pi$
of~$\projsp^3(\Z/q\Z)$, and $g$ takes the form $g=g_\pi$.  The index
$[\Gamma_0'(N):\Gamma_0'(Nq)]=|\projsp^3(\Z/q\Z)|$ follows from the
bijectivity.
\end{proof}

Next we study $\Gamma_0'(N)\bs\KN$.
Let $\GamsupzeroN$ denote the subgroup of~$\SLtwoZ$ defined by the
condition $b=0\mymod N$.  We show that $\GamsupzeroN\bs\SLtwoZ$
parametrizes $\Gamma_0'(N)\bs\KN$.

\begin{proposition}
For each matrix $g_o=\smallmatabcd$ of~$\SLtwoZ$, define a
corresponding matrix
$$
g=\iota(1_2,\smallmat100Ng_o\smallmat100{1/N}) 
=\left[\begin{array}{cc|cc}
1&0&0&0\\ 0&a&0&b/N \\ 
\hline
0&0&1&0 \\ 0&cN&0&d
\end{array}\right].
$$
Then $g$ lies in~$\KN$, and the map
$\GamsupzeroN g_o\mapsto\Gamma_0'(N)g$ from $\GamsupzeroN\bs\SLtwoZ$
to $\Gamma_0'(N)\bs\KN$ is well defined and bijective.
Thus $[\KN:\Gamma_0'(N)]=[\SLtwoZ:\GamsupzeroN]=N\prod_{p\mid N}(1+1/p)$.
\end{proposition}

\begin{proof}
One readily checks that $g\in\KN$ and that the map
$\GamsupzeroN g_o\mapsto\Gamma_0'(N)g$ is well defined and injective.
For surjectivity, consider any element~$h$ of~$\KN$ that does not lie
in~$\Gamma_0'(N)$, and let $h\inv$ have last column
$(*,\beta_o/N,*,\delta_o)'$.
Note that $\beta_o\ne0$, because otherwise $h\inv$ lies
in~$\Gamma_0'(N)$ and hence so does~$h$.
Let $\epsilon=\gcd(\beta_o,\delta_o)$.  There is an element
$h_o=\smallmat{\alpha_o}{\beta_o/\epsilon}{\gamma_o}{\delta_o/\epsilon}$
of~$\SLtwoZ$; let $g_o=h_o\inv
=\smallmat{\delta_o/\epsilon}{-\beta_o/\epsilon}{-\gamma_o}{\alpha_o}$,
and let $g$ correspond to~$g_o$ as above.  The $(2,4)$-entry of~$gh\inv$
is the inner product of $(0,\delta_o/\epsilon,0,-\beta_o/(\epsilon N))$
and $(*,\beta_o/N,*,\delta_o)$, which is~$0$.  Thus
$gh\inv\in\KN\cap\SptwoZ=\Gamma_0'(N)$ and so $\Gamma_0'(N)h=\Gamma_0'(N)g$.
\end{proof}

Third, we study upper triangular representatives.
Theorem~1.4 of~\cite{py13} (see also \cite{i93})
specializes to show for squarefree~$Nq$
that $\SptwoQ=\KNq\PtwozeroQ$ where $\PtwozeroQ$ is the Siegel
parabolic group of rational symplectic matrices having $c$-block~$0$.
Furthermore, the proof of the theorem can be made algorithmic,
and doing so is especially easy for squarefree~$N$.  With
representatives $\lset g_{1i}\rset$ and $\lset g_{2j}\rset$ of
$\Gamma_0'(Nq)\bs\Gamma_0'(N)$ of $\Gamma_0'(N)\bs\KN$ at hand, we
thus have an algorithm to decompose the representatives
$\lset g_{ij}\rset=\lset g_{1,i}g_{2,j}\rset$ of $\Gamma_0'(Nq)\bs\KN$
as $\lset\kappa_{ij}u_{ij}\rset$ with each $\kappa_{ij}\in\KNq$ and
each $u_{ij}\in\PtwozeroQ$.
We made no attempt to optimize the computation of~$\kappa$ and~$\mu$,
and the expense of computing the $(1+q+q^2+q^3)N\prod_{p\mid N}(1+1/p)$
decompositions $g=\kappa u$ led us to use small values of~$q$ whenever
possible.  On the other hand, we needed $q$ large enough to make tracing
down hit the bulk of~$\StwoKN^+$.  We proceeded by a mixture of
experiment and feel, with the experimental attempts to span a large
subspace of~$\StwoKN^+$ time-consuming because tracing down ran slowly
even with parallel computing.
Our values of~$q$ for tracing down ranged from~$3$ to~$11$.

\medskip

The trace down operator from $\CFs\KNq$ to~$\CFs\KN$ is
$$
f\longmapsto\sum_{i,j}f\wtk{u_{ij}}.
$$
Fix $i$ and~$j$, and let $u_{ij}=u=\smallmat{d^*}b0d$.  By a
familiar computation, $f\wtk u$ has Fourier coefficients $\fc t{f\wtk u}
=(\det d)^{-k}\e(\tr(td'b))\fc{t[d']}f$,
and so the Fourier coefficients of the trace down image are
$$
\fc t{{\rm TrDn}f}=\sum_{i,j}(\det d_{ij})^{-k}
\e(\tr(td_{ij}'b_{ij}))\fc{t[d_{ij}']}f.
$$

\section{Hecke Spreading\label{secHS}}

We used Hecke operators mainly to create elements of the Fricke minus
space $\SfourKN^-$ from elements of the plus space.  For this section,
let $G$ denote the subspace $\GritJtwoNcusp$ of~$\StwoKN^+$.
A Hecke operator~$T$ of $\SfourKN$ need not respect ring structure,
and so even though $T(G)\subset G$ and $G\cdot G\subset\SfourKN^+$,
the space $T(G\cdot G)$ need not lie in~$\SfourKN^+$.
Indeed, for prime divisors~$p$ of~$N$, the operator
$T(p^2)$ can take elements of $G\cdot G$ into~$\SfourKN^-$.
Here the operators $T(n)$ for $n$ coprime to~$N$ are standard
(\cite{py15}), and we let $T(p)$ and $T(p^2)$ for $p\mid N$ denote
the operators $T_{0,1}$ and $T_{1,0}$ of \cite{robertsschmidt07};
explicit single coset decomposition formulas for these operators
appear in \cite{pykeight}.
We increased our span of~$\SfourKN$ with $T(G\cdot G)$ for various~$T$.

Our computations represented elements of~$\SfourKN^\pm$ as vectors of
Fourier coefficients indexed by $\GampmsupzeroN$-equivalence classes
in~$\XtwoN$, with a cap on the determinants  of the class
representatives $\smallmat n{r/2}{r/2}{mN}$, \ie, $4nmN-r^2\le d$.
The cap was determined from the results of Jacobi restriction
on~$\SfourKN^+$, the larger of the two Fricke eigenspaces for our
range of $N$-values.
However, for a Hecke operator $T(n)$ to return a Fourier coefficient
vector indexed by such a determinant-shell of class representatives,
it needs for its input a Fourier coefficient vector indexed by class
representatives out to determinant~$n^2d$.  Depending on various
parameters, this can raise the vector length from hundreds to hundreds
of thousands, or even well over a million.  Computing a basis of
$G\cdot G$ to so many terms was a significant computational expense,
generally limiting our Hecke spreading to $T(n)$ for~$n=4,8,9,12$.
For $N$ coprime to~$6$, Hecke spreading was not available to us as a
method to span any of~$\SfourKN^-$.  Also, Hecke spreading had little
to start with for levels~$N$ at which $G$ is small, and indeed it had
nothing to start with when~$G=0$.  Furthermore, Hecke spreading into
$\SfourKN^+$ can produce only forms having Atkin--Lehner signatures
already in~$G\cdot G$, and Hecke spreading seems unable to reach old
forms in either Fricke eigenspace when they come from minus forms.
When tracing down and Hecke spreading didn't give us enough
dimensions, we searched for Borcherds products in the space, usually
in the Fricke minus space but sometimes in the Fricke plus space as
well.  We next proceed to a discussion of these matters.

\section{Borcherds Products\label{secBP}}

The following theorem gives sufficient conditions for a Borcherds product
to be a Siegel paramodular Fricke eigenform; it is a special case of
Theorem~3.3 of \cite{gpy15}, which in turn is quoted from
\cite{gn97,gn98,grit12} and relies on the work of R.~Borcherds.  The
corollary to follow will give sufficient conditions for such a
Borcherds product furthermore to be a cusp form when its level is
squarefree.

\begin{theorem}\label{BPthm}
Let $N$ be a positive integer.  Let $\psi\in\JzeroNwh$ be a weakly
holomorphic weight~$0$, degree~$N$ Jacobi form, having Fourier expansion
$$
\psi(\tau,z)=\sum_{\substack{n,r\in\Z\\n\gg-\infty}}c(n,r)q^n\zeta^r
\quad\text{where }q=\e(\tau),\ \zeta=\e(z).
$$
Define
\begin{alignat*}2
A&=\frac1{24}c(0,r)+\frac1{12}\sum_{r\in\Zpos}c(0,r),
 &\quad 
B\phantom{_0}&=\frac12\sum_{r\in\Zpos}r\,c(0,r),\\
C&=\frac12\sum_{r\in\Zpos}r^2c(0,r),
 &\quad
D_0&=\sum_{n\in\Zneg}\sigma_0(|n|)c(n,0).
\end{alignat*}
Suppose that the following conditions hold:
\begin{enumerate}
\item $c(n,r)\in\Z$ for all integer pairs $(n,r)$ such that $4nN-r^2\le0$,
\item $A\in\Z$,
\item $\sum_{j\in\Zpos}c(j^2nm,jr)\ge0$ for all primitive integer
  triples $(n,m,r)$ such that $4nmN-r^2<0$ and $m\ge0$.
\end{enumerate}
Then for weight $k=\frac12c(0,0)$ and Fricke eigenvalue
$\epsilon=(-1)^{k+D_0}$, the Borcherds product $\BL(\psi)$ lies
in~$\MFs\KN^\epsilon$.
For sufficiently large $\lambda$, for $\Omega=\smallmat\tau
zz\omega\in\UHPtwo$ and $\xi=\e(\omega)$, the Borcherds product has
the following convergent product expression on the subset
$\lset\Im \Omega>\lambda I_2\rset$ of~$\UHPtwo$:  
$$
\BL(\psi)(\Omega)=
q^A \zeta^B \xi^C 
\prod_{\substack{n,m,r\in \Z,\ m\ge0\\
\text{if $m=0$ then $n\ge0$}\\
\text{if $m=n=0$ then $r<0$}}}
(1-q^n\zeta^r\xi^{mN})^{c(nm,r)}.
$$
Also, let $\varphi(r)=c(0,r)$ for~$r\in\Znn$, and recall the
corresponding theta block,
$$
\TB(\varphi)(\tau,z)=\eta(\tau)^{\varphi(0)}
\prod_{r\in\Zpos}(\vartheta_r(\tau,z)/\eta(\tau))^{\varphi(r)}
\qquad\text{where }\vartheta_r(\tau,z)=\vartheta(\tau,rz).
$$
On $\lset\Im \Omega>\lambda I_2\rset$ the Borcherds product is a
rearrangement of a convergent infinite series,
$$
\BL(\psi)(\Omega)
=\TB(\varphi)(\tau,z)\xi^C\exp\left(-\Grit(\psi)(\Omega)\right).  
$$
\end{theorem}

We make some remarks about the theorem.
\begin{itemize}
\item $A$, $B$, $C$, and $D_0$ are finite sums.  Indeed, if the
  Fourier series expansion of $\psi$ is supported on~$n\ge n_o$ then
  it is supported on $4nN-r^2\ge4n_oN-N^2$, so especially $c(0,r)=0$
  if $r^2>N^2-4n_0N$.  Also, conditions (1) and~(3) in the theorem are
  finite to verify, (1) because the coefficients $c(n,r)$ where
  $4nN-r^2\le0$ are determined by the coefficients such that
  furthermore $n\le N/4$ and $|r|\le N$, and (3) because $\psi$ is
  supported on the pairs $(n,r)$ such that $4nN-r^2\gg-\infty$.
\item Condition (1) implies that $c(n,r)\in\Z$ for all $(n,r)$; this
  a general fact about weakly holomorphic Jacobi forms of weight~$0$.
  Because the coefficients are integral but possibly negative, the
  theta block in the theorem could be a theta block with denominator.
\item The quantity $A$ is often written as $(1/24)\sum_{r\in\Z}c(0,r)$
  and the condition $A\in\Z$ is often phrased that
  $\sum_{r\in\Z}c(0,r)$ is an integer multiple of~$24$.  We phrased
  the theorem to make clear that the $q^A$ in the product expression
  of $\BL(\psi)$ can be read off from the theta block in the series
  representation, as can the~$\zeta^B$.  The integrality of the
  coefficients and the condition $A\in\Z$ make $k$ integral because
  $c(0,0)+2\sum_{r\in\Zpos}c(0,r)$ is a multiple of~$24$, and so
  $c(0,0)$ is even.
\item The divisor of $\BL(\psi)$ is a sum of Humbert surfaces with
  multiplicities, the multiplicities necessarily nonnegative for
  holomorphy.  Let $\KN^+$ denote the supergroup of~$\KN$ obtained by
  adjoining the paramodular Fricke involution.  The sum in item~(3) is
  the multiplicity of the following Humbert surface in the divisor,
  $$
  {\rm Hum}(r^2-4nm,r)=\KN^+\lset\Omega\in\UHPtwo:
  \ip\Omega{\smallmat{n}{r/2}{r/2}{mN}}=0\rset.
  $$
  This surface lies in~$\KN^+\bs\UHPtwo$.  As the notation in the
  display suggests, this surface depends only on the discriminant
  $d=r^2-4nm$ and on~$r$, and furthermore it depends only on the
  residue class of~$r$ modulo~$2N$; this result is due to Gritsenko
  and Hulek \cite{gh98}.  We use it to parametrize Humbert surfaces as
  ${\rm Hum}(d,r)$, taking for each such surface a suitable
  $\smallmat{n}{\tilde r/2}{\tilde r/2}{mN}$ with $\gcd(n,m,\tilde r)=1$
  and $m\ge0$ and $\tilde r^2-4nmN=d$ and $\tilde r=r\mymod2N$.
\item The series representation of $\BL(\psi)$ in the theorem gives an
  experimental algorithm for the construction of holomorphic Borcherds
  products, based on using the series to calculate an initial portion of
  the Fourier--Jacobi expansion of $\BL(\psi)$.
  The algorithm and computer program for searching for suitable $\psi$
  will be covered in an article being prepared \cite{pyfindbps}; we
  used a comparatively simple part of the algorithm and program for
  the computation being reported in this article.  Our website for
  this article certifies the Borcherds products that we constructed,
  as will be discussed below.
\item The condition $\epsilon = (-1)^{k+D_0}$ where $k$ is the weight
  says in particular that for even~$k$ the Fricke eigenvalue of
  $\BL(\psi)$ is~$1$ if $\psi$ has principal part~$0$ and is~$-1$ if
  $\psi$ has principal part $1/q$.  In our computation we used such
  $\psi$ to create Fricke plus and minus eigenforms in spanning the
  spaces $\SfourKN$.
\end{itemize}

Theorem~\ref{BPthm} and Proposition~\ref{cuspidalityprop} let us determine
when a given weakly holomorphic Jacobi form of weight~$0$ and
squarefree index~$N$ gives rise through its Borcherds product to a cusp
form.

\begin{corollary}\label{BPthmcor}
Let $N$ be a squarefree positive integer, let $k$ be an odd positive
integer or let $k$ be one of $\lset2,4,6,8,10,14\rset$,
and let $\epsilon=\pm1$.
Let $\psi$, $A$, $B$, $C$, $D_0$ be as in Theorem~\ref{BPthm},
with $k=\frac12c(0,0)$ and $\epsilon=(-1)^{\tfrac12c(0,0)+D_0}$.
Suppose that the three conditions of the
theorem hold,
\begin{enumerate}
\item $c(n,r)\in\Z$ for all integer pairs $(n,r)$ such that $4nN-r^2\le0$,
\item $A\in\Z$,
\item $\sum_{j\in\Zpos}c(j^2n_om_o,jr_o)\ge0$ for all primitive integer
  triples $(n_o,m_o,r_o)$ such that $4n_om_oN-r_o^2<0$ and $m_o\ge0$.
\end{enumerate}
If $k$ is odd or $k=2$, or if $k\in\lset4,6,8,10,14\rset$ and~$C>0$,
then the Borcherds product $\BL(\psi)$ of Theorem~\ref{BPthm} lies in
the cusp form space $\SkKN^\epsilon$ and its first nonzero
Fourier-Jacobi coefficient has index~$C>0$.
\end{corollary}

\begin{proof}
The Borcherds product $\BL(\psi)$ lies in $\MFs\KN^\epsilon$.
The product expression and the series representation of $\BL(\psi)$
both show that its first nonzero Fourier--Jacobi coefficient has index~$C$.
If $k$ is odd or~$k=2$ then Proposition~\ref{cuspidalityprop} says
that $\BL(\psi)$ is a cusp form.
If $k\in\lset4,6,8,10,14\rset$ and~$C>0$ then because the Borcherds
product is a multiple of $q^A\zeta^B\xi^C$, the constant term
$\fc0{\BL(\psi)}$ is~$0$, and again Proposition~\ref{cuspidalityprop}
says that $\BL(\psi)$ is a cusp form.
If $\BL(\psi)$ is a cusp form then the index matrix $\smallmat A{B/2}{B/2}C$
of its term $q^A\zeta^B\xi^C$ must be positive, requiring~$C>0$.
\end{proof}

To create weight~$4$ Borcherds products for the computation being
reported in this article, we used two constructions of Jacobi forms
$\psi\in\JzeroNwh$ that are amenable to Corollary~\ref{BPthmcor}.  For
$\SfourKN^+$ Borcherds products, we used quotients
$\psi=\psi_{12}/\Delta_{12}$ where $\psi_{12}\in\JkNcusp{12}N$ and
$\Delta_{12}\in\JkNcusp{12}0$ is the classical discriminant function.
For each level~$N$ at which we used such Borcherds products
$\BL(\psi)$, a ``BP+'' file at our website \cite{yuen16b} gives the
singular part of each~$\psi$ that we used, so that the conditions of
Corollary~\ref{BPthmcor} can be verified for it; also the file gives
the leading Fourier--Jacobi coefficient $\phi$ of each Borcherds lift,
and the file describes each $\psi$ as a linear combination of a basis
of~$\JkNcusp{12}N$, giving the linear combination vector and then
giving the basis.
For $\SfourKN^-$ Borcherds products $\BL(\psi)$, a construction
$\phi\mapsto\phi|V_2/\phi$ for $\phi\in\JkNcusp4N$, where $V_2$ is an
index-raising Hecke operator \cite{ez85}, also contributes to~$\psi$;
this construction is central in \cite{gpy15}, and it gives a weakly
holomorphic Jacobi form with integral Fourier coefficients when $\phi$
is a theta block without denominator.  For each level~$N$ at which we
used such Borcherds products, a ``BP-'' file at our website again
gives the singular part of each $\psi$, then $\phi$, which is again the
leading Fourier--Jacobi coefficient of the Borcherds product, then the
linear combination vector, and then the basis.

\section{Construction of the nonlift newforms\label{secL249295}}

At level $N=249$ and at level $N=295$, Jacobi restriction gives a
truncated Fourier--Jacobi expansion of a putative weight~$2$ Fricke plus
space nonlift~$b_N$ with first Fourier--Jacobi coefficient~$0$, \ie,
$b_N(\Omega)=\phi_2(b_N)(\tau,z)\xi^{2N}+\phi_3(\tau,z)\xi^{3N}+\cdots$.
Because the first Fourier--Jacobi coefficient is~$0$, this $b_N$ is a
nonlift in~$\StwoKN^+$ if it exists in~$\StwoKN^+$ at all.  We hope to
show that $b_N$ exists in~$\StwoKN^+$ by showing that it takes the form
$b_N=\BL(\psi_N)$ for some $\psi_N\in\JzeroNwh$, though there is no
guarantee that this happens even if $b_N$ exists in~$\StwoKN^+$.
The infinite series form of $\BL(\psi_N)$ in Theorem~\ref{BPthm} shows
that necessarily $\psi_N=-\phi_3/\phi_2$.  Thus we make an educated
guess at an element $\psi_N$ of~$\JzeroNwh$ that has the same initial
Fourier expansion as $-\phi_3/\phi_2$ would if $\phi_3$ were divisible
by~$\phi_2$.  If $b_N$ indeed exists as a Borcherds product $\BL(\psi_N)$
and we guess~$\psi_N$ well then Corollary~\ref{BPthmcor} confirms that
$b_N$ exists in $\StwoKN^+$, and Jacobi restriction has helped us find
our desired nonlift.
An article being prepared on general methods for constructing nonlifts
\cite{pyinflation} will give details on how to identify a candidate
weakly holomorphic weight~$0$ Jacobi form as arising from a quotient
of dilated theta functions that demonstrably lies in~$\JzeroNwh$.  In
this article we simply present $\psi_N$ as a product of quotients of
dilated theta functions for $N=249,295$.

Consider the following product of quotients of dilated theta functions:
$$
\psi_{249}=\frac{\vartheta_8}{\vartheta_1}
\frac{\vartheta_{18}}{\vartheta_6}\frac{\vartheta_{14}}{\vartheta_7}\,.
$$
Each quotient in the previous display takes the form
$\vartheta_d/\vartheta_e$ with $e\mid d$; this divisibility makes the
quotient a weakly holomorphic Jacobi form of weight~$0$ and index
$(d^2-e^2)/2$, and thus altogether $\psi_{249}\in\JkNwh0{249}$.
The singular part of $\psi_{249}$ up to $q^{\lfloor249/4\rfloor}$,
which determines all of the singular part, is
\begin{align*}
4
&+\zeta^{-13}+\zeta^{-12}+\zeta^{-11}+\zeta^{-10}+\zeta^{-9}
 +\zeta^{-8}+2\zeta^{-7}+3\zeta^{-6}
 +2\zeta^{-5}+2\zeta^{-4}\\
&+2\zeta^{-3}+2\zeta^{-2}+3\zeta^{-1}+3\zeta+2\zeta^2
 +2\zeta^3+2\zeta^4+2\zeta^5+3\zeta^6+2\zeta^7+\zeta^8\\&
 +\zeta^9+\zeta^{10}+\zeta^{11}+\zeta^{12}+\zeta^{13}
 +q^2(\zeta^{45}+\zeta^{-45})
 +q^5(\zeta^{71}+\zeta^{-71})\\
&+q^6(\zeta^{78}+\zeta^{-78})
 +q^7(\zeta^{84}+\zeta^{-84})
 +q^8(\zeta^{90}+\zeta^{-90})\\
&+q^{24}(\zeta^{155}+\zeta^{-155})
 +q^{26}(\zeta^{161}+\zeta^{-161})
 +q^{28}(\zeta^{167}+\zeta^{-167})\\
&-q^{43}(\zeta^{207}+\zeta^{-207})
 +q^{54}(\zeta^{232}+\zeta^{-232}).
\end{align*}
This $\psi_{249}$ satisfies the three conditions of Corollary~\ref{BPthmcor}:
the Fourier coefficients of the singular part are integers; $A=2$; and
Table~\ref{tbl:249} shows that the Humbert surface multiplicity
$m_{d,r}$ is nonnegative for all possible $(d,r)$ such that any
term in the $(d,r)$-multiplicity formula could be nonzero.
Also we have $B=63$ and $C=498$ and $D_0=0$.
Because $c(0,0)=4$, the corollary says that $\BL(\psi_{249})$ lies
in~$\StwoKlevel{249}^+$, and its first nonzero Fourier--Jacobi
coefficient has index $C=2\cdot249$.

\begin{table}
\caption{Humbert surface multiplicities of $\BL(\psi_{249})$}
\centering
\begin{tabular}{|c|c|}\hline
$(d,r)$ & $m_{d,r}$\\
\hline\hline
(1,1)   & 22\\\hline
(1,167) & 10\\\hline
(3,2)   & 10\\\hline
(4,164) &  4\\\hline
(9,3)   &  7\\\hline
(12,192)&  1\\\hline
(16,4)  &  4\\\hline
(16,170)&  1\\\hline
(21,207)&  0\\\hline 
\end{tabular}
\quad
\begin{tabular}{|c|c|}\hline
$(d,r)$  & $m_{d,r}$\\
\hline\hline
(25,5)   & 3\\\hline
(25,161) & 1\\\hline
(33,45)  & 2\\\hline
(36,6)   & 4\\\hline
(40,232) & 1\\\hline
(49,7)   & 2\\\hline
(61,71)  & 1\\\hline
(64,8)   & 1\\\hline
(81,9)   & 1\\\hline 
\end{tabular}
\quad
\begin{tabular}{|c|c|}\hline
$(d,r)$ & $m_{d,r}$\\
\hline\hline
(84,84)   & 1\\\hline
(100,10)  & 1\\\hline
(108,78)  & 1\\\hline
(121,11)  & 1\\\hline
(121,155) & 1\\\hline
(132,90)  & 1\\\hline
(144,12)  & 1\\\hline
(169,13)  & 1\\\hline 
&\\\hline
\end{tabular}
\label{tbl:249}
\end{table}

With the knowledge that running Jacobi restriction to five terms gives
us a basis of~$\StwoKlevel{249}$, we can use the expansions of the
basis to compute the action of the Hecke operator~$T(2)$.  This
operator separates the space into one-dimensional eigenspaces, with
the space spanned by the nonlift eigenform~$f_{249}$ readily
identifiable because it doesn't lie in $\GritJkNcusp2{249}$.  We have
the first two Jacobi coefficients of~$\BL(\psi_{249})$, and this is
sufficient to express $f_{249}$ as a linear combination of
$\BL(\psi_{249})$ and a Gritsenko lift,
\begin{align*}
f_{249}&=
14\,\BL(\psi_{249})\\
&\quad-6\,\Grit(\TB(2; 2,3,3,4,5,6,7,9,10,13))\\
&\quad-3\,\Grit(\TB(2;2,2,3,5,5,6,7,9,11,12 ))\\
&\quad+3\,\Grit(\TB(2;1,3,3,5,6,6,6,9,11,12 ))\\
&\quad+2\,\Grit(\TB(2; 1,1,2,3,4,5,6,9,10,15))\\
&\quad+7\,\Grit(\TB(2; 1,2,3,3,4,5,6,9,11,14)).
\end{align*}
Here $\TB(2; 2,3,3,4,5,6,7,9,10,13)$ is the weight~$2$ theta block
$\TB(\varphi)$ where $\varphi(r)=1$ for $r=2,4,5,6,7,9,10,13$ and
$\varphi(3)=2$ and $\varphi(r)=0$ for all other~$r\in\Zpos$, and
similarly for the other theta blocks.
Many Fourier coefficients of~$f_{249}$ are given at our website \cite{yuen16b}.
Thus $f_{249}$ is congruent to a lift modulo~$14$, and as noted
earlier in this article, the two isogenous Jacobians of genus~$2$
curves of conductor~$249$ defined over~$\Q$ have torsion groups
$\Z/14\Z$ and~$\Z/28\Z$.

Because $\StwoKlevel3=0$, and because $\StwoKlevel{83}$ is spanned by
Gritsenko lifts and level-raising operators take lifts to lifts, the
nonlift eigenform $f_{249}$ is in fact a newform.

To compute the $T(n)$-action on~$f_{249}$ for~$n\le25$ coprime
to~$249$, we expanded~$f_{249}$ to~$25$ Fourier--Jacobi coefficients.
The relevant eigenvalues are
$\lambda_2=-2$ and $\lambda_4=0$,
$\lambda_5=0$ and $\lambda_{25}=-3$,
$\lambda_7=-1$,
$\lambda_{11}=1$,
$\lambda_{13}=0$,
$\lambda_{17}=-1$,
$\lambda_{19}=-6$,
$\lambda_{23}=0$.
These eigenvalues determine the spin $p$-Euler factor for $p=2,5$,
because in general for $p\nmid N$ the spin $p$-Euler factor for
weight~$2$ is (see section 5.3~of \cite{andrianov09}, for example)
$$
1-\lambda_pT+(\lambda_p^2-\lambda_{p^2}-1)T^2-\lambda_ppT^3+p^2T^4.
$$
The spin $2$-Euler factor and the spin $5$-Euler factor of~$f_{249}$ are
$$
1+2T+3T^2+4T^3+4T^4,\qquad1+2T^2+25T^4,
$$
and the reader can check at the online database {\tt lmfdb.org} that these
Euler factors match those of the genus~$2$ curve class 249.a.
The Jacobian of this curve is an abelian surface of conductor~$249$
defined over~$\Q$, and it has the same Euler factors.
The spin $p$-Euler factors of~$f_{249}$ for $p=7,11,13,17,19,23$ are
also compatible with the 249.a $L$-factors at {\tt lmfdb.org}, with a
match depending on the uncomputed $\lambda_{p^2}$ in each case.

Similarly, consider
$$
\psi_{295}=\frac{\vartheta_{12}}{\vartheta_3}
\frac{\vartheta_{15}}{\vartheta_5}\frac{\vartheta_{12}}{\vartheta_6}
\frac{\vartheta_{14}}{\vartheta_7}\,.
$$
Here $\psi_{295}$ lies in~$\JkNwh0{295}$, and its singular part up to
$q^{\lfloor295/4\rfloor}$ is
\begin{align*}
4
&+\zeta^{-16}+\zeta^{-13}+\zeta^{-11}+2\zeta^{-10}
 +\zeta^{-9}+\zeta^{-8}+2\zeta^{-7}+2\zeta^{-6}
 +2\zeta^{-5}+2\zeta^{-4}\\
&+3\zeta^{-3}+2\zeta^{-2}+2\zeta^{-1}+2\zeta+2\zeta^2
 +3\zeta^3+2\zeta^4+2\zeta^5+2\zeta^6+2\zeta^7+\zeta^8\\
&+\zeta^9+2\zeta^{10}+\zeta^{11}+\zeta^{13}+\zeta^{16}
 +q^4(\zeta^{69}+\zeta^{-69})
 -q^7(\zeta^{91}+\zeta^{-91})\\
&+q^{10}(\zeta^{109}+\zeta^{-109})
 +q^{11}(\zeta^{114}+\zeta^{-114})
 +q^{12}(\zeta^{119}+\zeta^{-119})\\
&+q^{25}(\zeta^{172}+\zeta^{-172})  
 +q^{27}(\zeta^{179}+\zeta^{-179})
 +q^{28}(\zeta^{182}+\zeta^{-182})\\
&+2q^{29}(\zeta^{185}+\zeta^{-185})
 -q^{37}(\zeta^{209}+\zeta^{-209})
 +q^{41}(\zeta^{220}+\zeta^{-220})\\
&+q^{42}(\zeta^{223}+\zeta^{-223})
 +q^{44}(\zeta^{228}+\zeta^{-228})
 +q^{46}(\zeta^{233}+\zeta^{-233})\\
&+q^{59}(\zeta^{264}+\zeta^{-264}).
\end{align*}
This $\psi_{295}$ satisfies the three conditions of Corollary~\ref{BPthmcor},
similarly to~$\psi_{249}$; here the Humbert
surface multiplicities are given in Table \ref{tbl:295}.
In this case, $A=2$ and $B=68$ and $C=590$ and $D_0=0$.
Again $c(0,0)=4$, so $\BL(\psi_{295})$ lies
in~$\StwoKlevel{295}^+$ and its first nonzero Fourier--Jacobi
coefficient has index $C=2\cdot295$.

\begin{table}
\caption{Humbert surface multiplicities of $\BL(\psi_{295})$}
\centering
\begin{tabular}{|c|c|}\hline
$(d,r)$  & $m_{d,r}$\\
\hline\hline
(1,1)    &22\\\hline
(1,119)  &10\\\hline
(4,2)    &10\\\hline
(4,238)  & 4\\\hline
(5,185)  & 3\\\hline
(9,3)    & 6\\\hline
(9,233)  & 2\\\hline
(16,4)   & 4\\\hline
(16,114) & 2\\\hline
(20,220) & 1\\\hline
\end{tabular}
\quad
\begin{tabular}{|c|c|}\hline
$(d,r)$  & $m_{d,r}$\\
\hline\hline
(21,91)  & 0\\\hline
(21,209) & 0\\\hline
(25,5)   & 4\\\hline
(36,6)   & 2\\\hline
(41,69)  & 1\\\hline
(49,7)   & 2\\\hline
(64,8)   & 2\\\hline
(64,228) & 1\\\hline
(76,264) & 1\\\hline
(81,9)   & 1\\\hline
\end{tabular}
\quad
\begin{tabular}{|c|c|}\hline
$(d,r)$   & $m_{d,r}$\\
\hline\hline
(81,109)  & 1\\\hline
(84,172)  & 1\\\hline
(84,182)  & 1\\\hline
(100,10)  & 2\\\hline
(121,11)  & 1\\\hline
(169,13)  & 1\\\hline
(169,223) & 1\\\hline
(181,179) & 1\\\hline
(256,16)  & 1\\\hline
          &  \\ \hline
\end{tabular}
\label{tbl:295}
\end{table}

The Hecke operator $T(2)$ separates $\StwoKlevel{295}$ into
one-dimensional eigenspaces, and we can identify the nonlift
eigenform~$f_{295}$ and then express it as a linear combination of
$\BL(\psi_{295})$ and a Gritsenko lift,
\begin{align*}
f_{295}&=
7\,\BL(\psi_{295})\\
&\quad+4\,\Grit(\TB(2;1,1,2,3,3,4,5,10,13,16 ))\\
&\quad+\Grit(\TB(2; 1,1,2,3,4,7,8,10,11,15))\\
&\quad-2\,\Grit(\TB(2;2,3,4,4,6,7,8,10,10,14 ))\\
&\quad-\Grit(\TB(2;1,1,5,6,7,8,8,9,10,13 ))\\
&\quad-3\,\Grit(\TB(2; 1,2,3,4,5,5,8,9,13,14)),
\end{align*}
and many Fourier coefficients of $f_{295}$ are at our website.
Here the nonlift eigenform is congruent to a lift modulo~$7$,
and one of the two isogenous Jacobians of genus~$2$ curves of
conductor~$295$ defined over~$\Q$ has torsion group $\Z/7\Z$, as shown
at {\tt lmfdb.org}.
The nonlift eigenform $f_{295}$ is a newform because
$\StwoKlevel5=\StwoKlevel{59}=0$.
This time, after expanding~$f_{295}$ to sixteen Fourier--Jacobi
coefficients, our computed Hecke eigenvalues of~$f_{295}$ are
$\lambda_2=-2$ and $\lambda_4=0$,
$\lambda_3=-1$ and $\lambda_9=0$,
$\lambda_7=1$,
$\lambda_{11}=2$,
$\lambda_{13}=-2$,
excluding~$\lambda_5$ because $5$ divides~$295$.
The spin $2$-Euler factor and the spin $3$-Euler factor of~$f_{295}$ are
$$
1+2T+3T^2+4T^3+4T^4,\qquad1+T+3T^3+9T^4,
$$
and these match the corresponding Euler factors of the genus~$2$ curve
isogeny class 295.a at {\tt lmfdb.org}.
The Jacobian of this curve is an abelian surface of conductor~$295$,
having the same Euler factors.
The spin $p$-Euler factors of~$f_{295}$ for $p=7,11,13$ are
also compatible with the 295.a $L$-factors at {\tt lmfdb.org}, with a
match depending on the uncomputed $\lambda_{p^2}$ in each case.

\bibliographystyle{alpha}
\bibliography{sqfree-submit}

\end{document}